\documentclass[a4paper,10pt]{amsart}
\usepackage[utf8x]{inputenc}
\usepackage{graphicx, xcolor}
\usepackage{amsmath,amssymb,amsthm,comment,mathtools}
\usepackage[colorlinks,allcolors=blue]{hyperref} 
\usepackage{cleveref}
\usepackage{pdfpages}
\usepackage{amsmath}
\usepackage{subfigure,cite}
\usepackage{thm-restate}

\usepackage[a4paper,margin=2.6cm]{geometry}

\newtheorem{thm}{Theorem}[section]
\newtheorem{cor}[thm]{Corollary}

\newtheorem{prop}[thm]{Proposition}
\newtheorem{lem}[thm]{Lemma}

\newtheorem{conj}[thm]{Conjecture}

\newtheorem{ex}[thm]{Example}
\newtheorem{defi}[thm]{Definition}

\newtheorem{question}[thm]{Question}

\def\GA{{\rm GA}(t,k)}
\def\GAsin{{\rm GA}(t,k)'}
\def\lk{{\rm link}}
\def\del{{\rm del}}

\title{Generalized Andrásfai graphs and special Betti diagrams of edge ideals}
 \author[S. Asensio]{Sara Asensio}
 \address{Instituto de Investigaci\'on en Matem\'aticas (IMUVa), Universidad de Valladolid, Valladolid, Spain}
 \email{sara.asensio@uva.es}

 \author[I. García-Marco]{Ignacio García-Marco}
 \address{Instituto de Matem\'aticas y Aplicaciones (IMAULL), Secci\'on de Matem\'aticas, Facultad de
Ciencias, Universidad de La Laguna, 38200, La Laguna, Spain}
 \email{iggarcia@ull.edu.es}

\author[P. Gimenez]{Philippe Gimenez}
\address{Instituto de Investigaci\'on en Matem\'aticas (IMUVa), Universidad de Valladolid, Valladolid, Spain}
 \email{pgimenez@uva.es}

\begin{document}

\begin{abstract}
Edge ideals of graphs were introduced by Villarreal in 1990, and have been the subject of many studies since then. In the same year, Fröberg characterized edge ideals with regularity 2 in combinatorial terms. This result was generalized by Fernández-Ramos and Gimenez to regularity 3 for bipartite graphs. A key ingredient in these results is the particular shape of the Betti diagrams of the edge ideals of the graphs obtained after removing a Hamiltonian cycle from either a complete graph $ K_k$ or a complete bipartite graph $K_{k,k}$.  

In this work, we consider the family of Generalized Andrásfai graphs ${\rm GA}(t,k)$ with $t\geq 1 $ and $k \geq 2$. This family extends the families of complete graphs, since $K_{k+1} = {\rm GA}(1,k)$, and complete bipartite $k$-regular graphs, since $K_{k,k} = {\rm GA}(2,k)$. We show that  the results known for $ K_k$ and $ K_{k,k}$ can be naturally extended to this family. More precisely, when removing a suitable Hamiltonian cycle from ${\rm GA}(t,k)$, the resulting edge ideal has regularity $t+2$, projective dimension $t(k-2)$ and a Betti diagram exhibiting a generalized version of the same special shape. 
\end{abstract}

\maketitle

\section{Introduction}

Monomial ideals are a classic subject of study in Commutative Algebra, both for their intrinsic interest and because they often capture important information about arbitrary homogeneous ideals through initial ideals. Among them, quadratic squarefree monomial ideals, known as edge ideals, provide one of the most fruitful bridges between Commutative Algebra and Graph Theory. Introduced by Villarreal in 1990 \cite{Vil90}, edge ideals have become a fundamental object of study since then (see, for instance, \cite{MillerSturmfels, HerzogHibi, Villarreal2026}).
Although we will focus on how Graph Theory can help us solve a problem concerning edge ideals, there are also results in the opposite direction (see, e.g., \cite{FHVT11, Eng25}).

A fundamental problem in the area is to understand the Betti numbers of an edge ideal in terms of combinatorial properties of the underlying graph. Since explicit formulas for all Betti numbers are rarely available (see the PhD Thesis of Jacques \cite{Jac04} for some examples), a more accessible goal is to determine the location of the nonzero entries in the Betti diagram. We will refer to this region as the {\em shape} of the Betti diagram. Thanks to Hochster's formula \cite{Hoc77}, these questions are often approached through homological methods, although alternative techniques such as splittings have also proved effective \cite{EK90,HVT07,HVT08,HVT22}.

One of the cornerstones in the study of the shape of Betti diagrams is Fröberg's theorem \cite{Fro90}, which states that the edge ideal of a graph has a $2$-linear resolution if and only if the complement graph has no induced cycles of length at least $4$. Later, Eisenbud, Green, Hulek and Popescu \cite{EGHP04} showed that, whenever the resolution is not linear, the first nonlinear syzygies occur in homological degree  $\ell-3$, where $\ell$ denotes the minimum length of an induced cycle of length at least $4$ in the complement graph. Fernández-Ramos and Gimenez  refined this result in \cite{FRG09} by  determining the Betti number $\beta_{\ell -3,\ell}$ combinatorially. Their approach relied on explicit formulas for the Betti numbers of edge ideals of complements of cycles, whose Betti diagrams exhibit a remarkable extremal behavior: the regularity is only attained at the last step of the resolution, and the last row contains a unique nonzero entry equal to $1$.

After this work, the same authors proved in \cite{FRG14} an analogous of Fröberg's Theorem for bipartite graphs. More precisely, they characterized  the edge ideals of bipartite graphs with regularity at most $3$ as those whose bipartite complement does not have an induced cycle of length at least $6$. They also described the first nonzero entry in the third row of the Betti diagram when it exists. Their approach relied on explicit formulas for the Betti numbers of edge ideals of bipartite complements of even cycles, whose Betti diagrams exhibit the same special shape.

These examples naturally lead to the following question: for each $r \geq 3$, does there exist a family of graphs $\mathcal F_r$ whose edge ideals have regularity $r$, arbitrarily large projective dimension, and the same extremal shape of the Betti diagram?
 In this language, the families $\mathcal F_3$ and $\mathcal F_4$ provided in previous works are complements of cycles for $r = 3$, and bipartite complements of even cycles for $r = 4$.
  The purpose of this paper is to answer this question affirmatively. We introduce a family of graphs arising from Generalized Andrásfai graphs that extends the previous ones and solves this problem for all values of $r \geq 3$.

 For $t\geq 1$ and $k \geq 2$, the Generalized Andrásfai graph $\GA$ is a $k$-regular circulant graph on $t(k-1)+2$ vertices (see Section \ref{sub:graph} for the definition, and Figure \ref{fig:GA(3,6)} for an example).  
 \begin{figure}[!ht]
 \centering
 	\includegraphics[width=5cm]{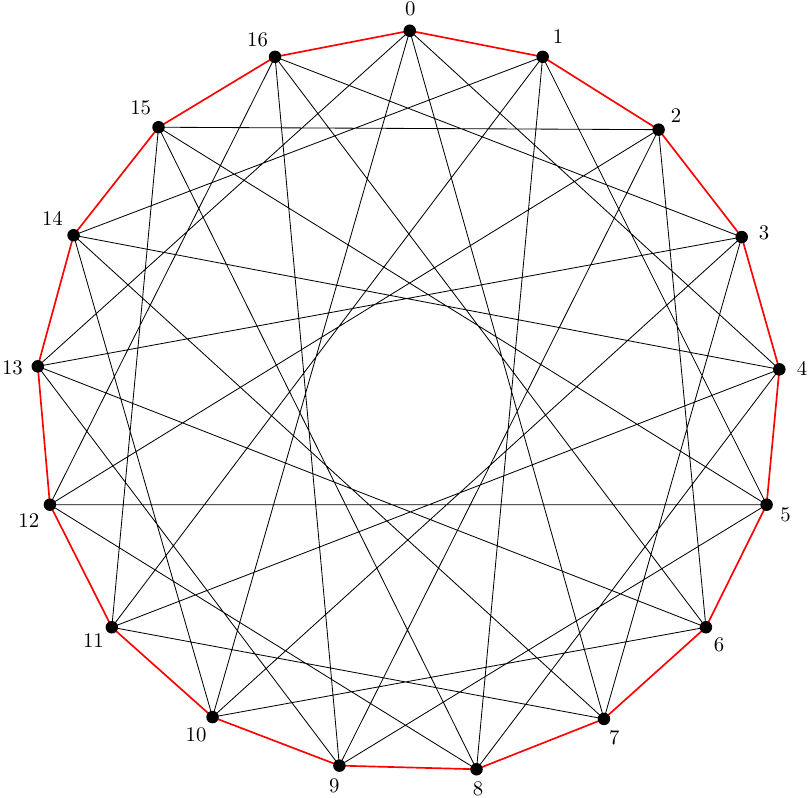}
    \caption{Generalized Andrásfai ${\rm GA}(3,6) = {\rm And}(6)$ with the exterior cycle marked in red}
 	\label{fig:GA(3,6)}
 \end{figure} 
This family was introduced by Das, Biswas and Saha \cite{DBS22} as an extension of the classical Andrásfai graphs \cite{And64} (which first appeared in the context of extremal combinatorics, as they are non-bipartite triangle-free graphs with the largest possible ratio between degree and number of vertices). Generalized Andrásfai graphs correspond to complete graphs when $t=1$ and balanced complete bipartite graphs when $t=2$.

Let ${\rm GA}(t,k)'$ denote the graph obtained from ${\rm GA}(t,k)$ by deleting the edges of a distinguished Hamiltonian cycle, which we call the exterior cycle.  Our main results show that, for every $k\geq3$, the regularity of the edge ideal of $\GAsin$ is $r := t+2$, its projective dimension is $p:= t(k-2)$, the regularity is only attained at the last step of the resolution and  $\beta_{p,p+r} = 1$ (see Figure \ref{fig:BettiDiagramk5} for the Betti tables of ${\rm GA}(t,5)'$ for $t \in \{1,2,3,4\}$). Consequently, for every $r\geq2$, the family
\[
\mathcal F_r=\{{\rm GA}(r-2,k)' \ |\ k\geq3\}
\]
has the desired properties. Moreover, this construction naturally recovers the previously known examples: when $r=3$ one obtains complements of cycles, while $r=4$ yields bipartite complements of even cycles.

\begin{figure}[htb]
\centering
	\includegraphics[width=14cm]{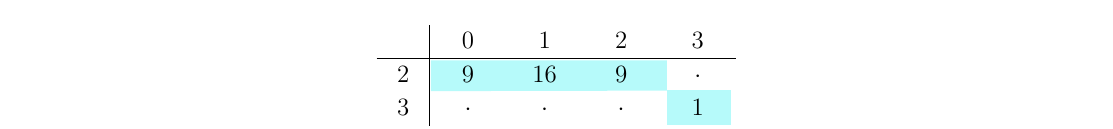}

	\includegraphics[width=14cm]{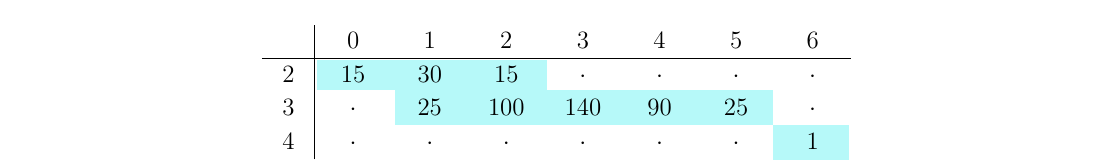}

\includegraphics[width=14cm]{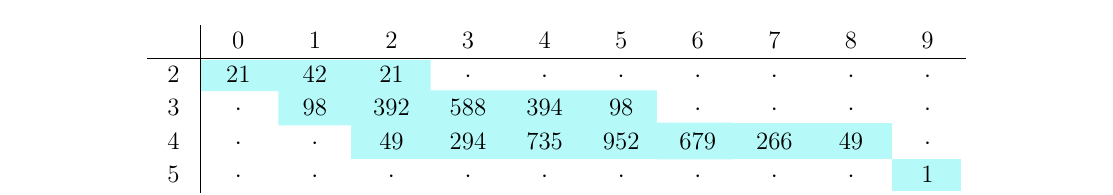}

\includegraphics[width=14cm]{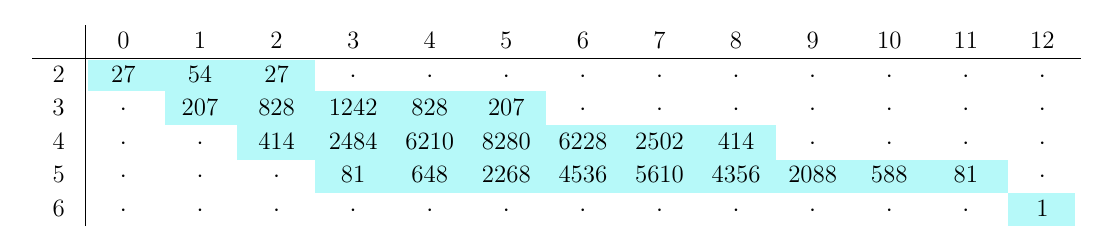}
   
  	\caption{From top to bottom: Betti diagrams of $I({\rm GA}(t,5)')$ for $t\in\{1,2,3,4\}$}
	\label{fig:BettiDiagramk5}
\end{figure}

An important feature of Generalized Andrásfai graphs with $t \geq 3$ is that deleting different Hamiltonian cycles may produce different shapes of Betti diagrams (see Figure \ref{fig:BettiDiagramAnd6}). Thus, the choice of the exterior cycle is essential for our results. This phenomenon does not occur for $t\in\{1,2\}$, as all Hamiltonian cycles are equivalent under automorphisms.

\begin{figure}[htb]
\centering
	\includegraphics[width=13.8cm]{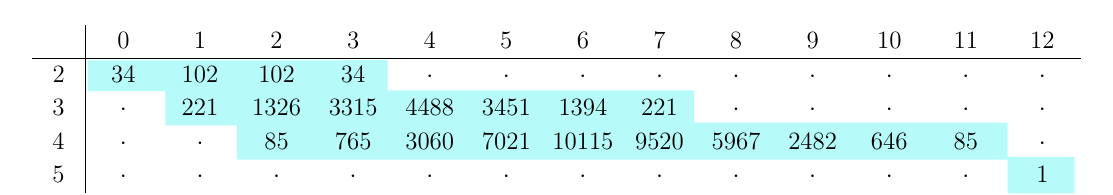}

\includegraphics[width=14cm]{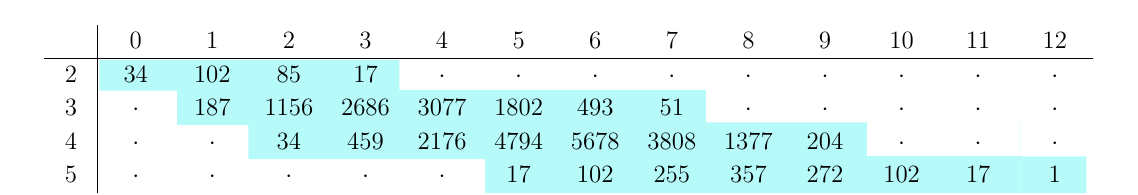}

\includegraphics[width=14cm]{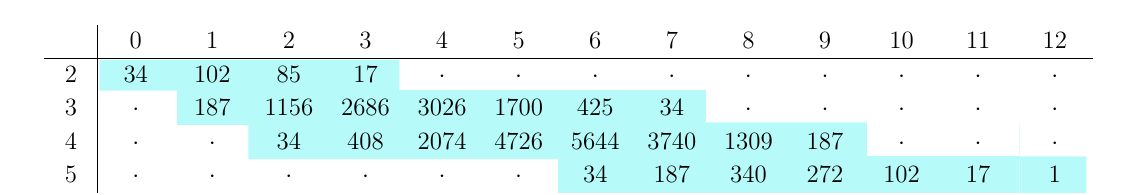}

  	\caption{From top to bottom: Betti diagrams of the edge ideals of ${\rm GA}(3,6)'$, ${\rm GA}(3,6)$ without the edges in the Hamiltonian cycle $C_4 = (4k \mod 17 \, \vert \, 0 \leq k \leq 16)$, and ${\rm GA}(3,6)$ without the edges in the Hamiltonian cycle $C_7 = (7k \mod 17 \, \vert \, 0 \leq k \leq 16)$}
	\label{fig:BettiDiagramAnd6}
\end{figure}

The paper is organized as follows.  In Section~\ref{sec:prel} we introduce the necessary graph-theoretic, algebraic and homological background. Next, we determine all Betti numbers in the linear strand of the edge ideals associated to  $\GAsin$. We then compute the induced matching number of these graphs, which allows us to determine the last nonzero number in the main diagonal of the corresponding Betti diagrams. In Sections \ref{sec:Regularity} and \ref{sec:ProjectiveDimension}, which require more algebraic and homological machinery, we compute the regularity and projective dimension of the edge ideals we are considering. In Section \ref{sec:t3} we completely describe the shape of the Betti diagram when $t = 3$. In the last section we present some open problems that have arisen while working on this article.

\section{Preliminaries}\label{sec:prel}

Since this work involves Combinatorics, Commutative Algebra and Homological Algebra, we start by introducing the concepts of these areas that will appear throughout the  paper.

\subsection{Graph theory}\label{sub:graph}
We  consider simple (without multiple edges or loops) undirected graphs. 
Some well-known families of graphs which will appear in this paper are: the {\em cycle} $C_n$ on $n$ vertices, the {\em complete graph} $K_n$ on $n$ vertices, and the {\em complete bipartite graph} $K_{a,b}$ on $a+b$ vertices, with bipartition sets of sizes $a$ and $b$. 

For a graph $G = (V(G),E(G))$, we denote its {\em complement graph} by $G^c$, that is, the graph with the same vertices, and where two distinct vertices are adjacent if and only if they are not adjacent in $G$.

The {\em degree} $\deg(v)$ of a vertex $v$ is the number of edges incident to $v$, and a graph $G$ is said to be $d$-regular when all its vertices have the same degree $d$. The set of vertices of $G$ which are adjacent to $v$ is denoted by $N(v)$, where the letter $N$ stands for the word {\em neighborhood}. Furthermore, we denote by $N[v]$ the {\em closed neighborhood} of $v$, which equals $N(v)\cup\{v\}$.

Given a subset $S$ of $V(G)$, we say it is {\em independent} if there are no edges in $G$ joining two vertices of $S$, whereas a complete subgraph of $G$ is called a {\em clique}. The {\em induced subgraph} of $G$ on $S$ is the graph whose vertices are the elements of $S$ and whose edges are all the edges of $G$ whose endpoints belong to $S$. When $W$ is an induced subgraph of $G$ we write $W\leq G$, or $W<G$ if it is proper. 

Given $S\subseteq V(G)$, we denote by $G\setminus S$ the induced subgraph of $G$ on $V(G)\setminus S$. Using this notation, we define the {\em vertex connectivity} of a connected graph $G$ as the smallest size $\kappa(G)$ of a subset $S\subseteq V(G)$ such that $G\setminus S$ is disconnected. The vertex connectivity of a graph is always upper bounded by the minimum degree of the graph, and when these two values coincide we say that the graph is {\em $\kappa$-optimal}.

An {\em induced matching of size $m$} of a graph $G$, denoted by $m\cdot K_2$, is an induced subgraph of $G$ consisting of $m$ disjoint edges. The size of the largest induced matching of $G$ is called its {\em induced matching number}.

In this paper we focus on the study of Generalized Andrásfai graphs, which are some specific circulant graphs (Cayley graphs of finite cyclic groups). 

Let $n > 0$, and consider the finite cyclic group $(\mathbb Z_n,+)$ and a subset $S\subseteq \mathbb Z_n$ called the {\em connection set}. The {\em Cayley graph} ${\rm Cay}(\mathbb Z_n,S)$ of $\mathbb Z_n$ with respect to $S$ is the graph whose vertex set is $\mathbb Z_n$ and whose edges are $\{a,\,a + s\}$, where $a\in \mathbb Z_n$ and $s\in S$.

Take $t \geq 1$ and $k \geq 2$ and denote $n := t(k-1)+2$. The {\em Generalized Andrásfai graph} $\GA$ with parameters $t$ and $k$ is the Cayley graph of $(\mathbb Z_{n},\,+)$ with respect to $\{1+j t\, |\, j\in\{0,1,\dots,k-1\}\}$. In other words, for $i, j \in \mathbb Z_{n}$ and $0 \leq i < j < n$, then
\[\{i,j\} \in E(\GA) \Longleftrightarrow\, j - i \equiv 1 \pmod t. \]
These graphs are vertex transitive, $k$-regular and triangle-free, and coincide with the complete graph $K_{k+1}$ when $t=1$, the complete bipartite graph $K_{k,k}$ when $t=2$, and the Andrásfai graph ${\rm And}(k)$ when $t=3$.  

The family of graphs that we study in this work are Generalized Andrásfai graphs from which the exterior cycle (i.e., edges $\{i,i+1\}$ with $i \in \mathbb Z_{n}$) has been removed. This graph will be denoted by $\GAsin$, and is the Cayley graph of $(\mathbb Z_{n},\,+)$ with respect to the set $\{1+jt\,|\,j\in\{1,2,\dots,k-2\}\}$ (see Figure \ref{fig:GA(3,6)}). In other words, 
for $i, j \in \mathbb Z_{n}$ and $0 \leq i < j < n$, then 
\begin{equation}\label{eq:GAsin} \{i,j\} \in E(\GAsin)\ \Longleftrightarrow \ j - i \equiv 1\ ({\rm mod}\ t) \text{ and } j - i \notin \{1,\,n-1\}. \end{equation}

\subsection{Commutative algebra} Let $\mathbb K$ be an arbitrary field. Given a graph $G=(V(G),E(G))$ with $V(G)=\{x_1,\dots,x_n\}$, we consider the polynomial ring $R=\mathbb K[x_1,\dots,x_n]$, where by abuse of notation the variables in the polynomial ring are identified with the vertices of the graph.

The {\em edge ideal} associated to $G$ in $R$ is $I(G)=\langle x_ix_j\, |\, \{x_i,x_j\}\in E(G)\rangle$. From the definition, it is clear that edge ideals are quadratic squarefree monomial ideals. Hence, they are homogeneous ideals and we can consider their minimal graded free resolutions:$$ 0\rightarrow \oplus_j R(-j)^{\beta_{p,j}}\rightarrow \oplus_j R(-j)^{\beta_{p-1,j}}\rightarrow\cdots\rightarrow \oplus_j R(-j)^{\beta_{0,j}}\rightarrow I(G) \rightarrow 0 \, .$$

Some invariants associated to these resolutions that we will study in this paper are:
\begin{itemize}
    \item The {\em graded Betti numbers} $\beta_{i,j}$, where $\beta_{i,j}$ is the number of generators of degree $j$ in the $i$-th free module of the resolution. 

    Graded Betti numbers are usually collected in a table called the {\em Betti diagram} in such a way that the entry in the $i$-th column and the $j$-th row equals $\beta_{i,i+j}$.
    
    For every edge ideal, $\beta_{0,2}=|E(G)|$ and $\beta_{0,j} = 0$ for all $j \neq 2$. Hence, the first column of its Betti diagram has only one nonzero entry. The row of this diagram containing the Betti numbers $\beta_{i,i+2}$ is usually called the {\em linear strand} of $I(G)$. By a result of Katzman \cite[Lemma 2.2]{Kat06}, $\beta_{i,d}=0$ for all $d > 2\,(i+1)$. The diagonal of the Betti diagram  containing the Betti numbers $\beta_{i,\,2(i+1)}$ will be called the {\em main diagonal}.

    \item The {\em projective dimension} ${\rm pd}(I(G))$, defined as ${\rm pd}(I(G))=\max\{i \,|\,\beta_{i,i+j}\neq 0\}$, is the length $p$ of the resolution or, equivalently, the label of the last nonzero column of the Betti diagram.
    \item The {\em (Castelnuovo-Mumford) regularity}, defined as ${\rm reg}(I(G))=\max\{j \,|\,\beta_{i,i+j}\neq 0\}$ or, equivalently, as the label of the last nonzero row of the Betti diagram. Whenever $\beta_{i, i+{\rm reg}(I(G))} \neq 0$ for some $i \in \mathbb N$, we say that the regularity is attained at the $i$-th step of the resolution. 
\end{itemize}

In this work we will study the previous invariants for $I(\GAsin)$. 

\subsection{Simplicial complexes}

% \nacho{Una cosa que me raya o ralla mucho es que a las homologías las solemos llamar $\widetilde{H}_i$ y a los números de Betti los solemos llamar $\beta_{i,i+j}$. Me molestan los subíndices, porque para los números de Betti la $i$ es el paso y la $J$ la fila en el diagrama de Betti y el subíndice de $\widetilde{H}_i$ va más relacionado con la fila del diagrama de Betti. Eso hace que me suene mejor llamar a las homologías $\widetilde{H}_j$. No sé si vale la pena cambiarlo, además hay que tener cuidado porque en la parte del dual de alexander, ahí el subíndice $i$ de $\widetilde{H}_i$ si que va ligado al $i$ de $\beta_{i,i+j}$.}

Simplicial complexes naturally appear when computing Betti numbers using homological tools, so we introduce here their definition, as well as some examples and special subcomplexes.

\begin{defi}\label{def:SimplicialComplex}
    Given a finite set $V$, a {\em simplicial complex} $\Delta$ with {\rm vertex set} $V$ is a family of subsets of $V$ such that: \begin{enumerate}
         \item $\{v\}\in\Delta$ for all $v\in V$.
         \item If $\sigma\in\Delta$ and $\tau\subseteq\sigma$, then $\tau\in\Delta$.
     \end{enumerate}
\end{defi}

The elements of a simplicial complex $\Delta$ are called {\em faces}. %and the maximal faces are called {\em facets}.
Given a subset $W$ of the set of vertices $V$ of $\Delta$, the {\em induced subcomplex} of $\Delta$ on $W$ is defined as the simplicial complex whose faces are the faces of $\Delta$ constituted by elements of $W$. Furthermore, given a face $\sigma$ of $\Delta$, we can consider two induced subcomplexes of $\Delta$ which are called {\em link} and {\em deletion} of the face $\sigma$ respectively:
\begin{itemize}
    \item $\lk_{\Delta}(\sigma):=\{\tau\in\Delta\, : \, \tau\cap\sigma=\emptyset,\, \tau\cup\sigma\in\Delta\}$.

    \item $\del_{\Delta}(\sigma):=\{\tau\in\Delta\,:\, \tau\cap\sigma=\emptyset\}$.
\end{itemize}

In addition, given two simplicial complexes $\Delta_1$ and $\Delta_2$ with disjoint vertex sets, the {\em join} of $\Delta_1$ and $\Delta_2$ is the simplicial complex $\Delta_1*\Delta_2:=\{\sigma\cup\tau\,:\,\sigma\in\Delta_1,\,\tau\in\Delta_2\}$ on $V(\Delta_1)\cup V(\Delta_2)$.

\begin{ex}
For a simple graph $G$, its {\em independence complex} $\Delta(G)$ is the simplicial complex with vertex set $V(G)$, and whose faces are the independent sets of vertices of $G$. For $v \in V(G)$, one has that  $\lk_{\Delta(G)}(v)=\Delta(G\setminus N[v])$ and $\del_{\Delta(G)}(v)=\Delta(G\setminus \{v\})$, where the independence complex of the graph with no vertices is $\{\emptyset\}$ by convention. Moreover, if $G$ and $H$ are two disjoint graphs, then $\Delta(G\sqcup H)=\Delta(G)*\Delta(H)$.
\end{ex}

The story of the connection between simplicial complexes and Commutative Algebra started with Stanley-Reisner theory. Given a simplicial complex $\Delta$, the {\em Stanley-Reisner ideal} associated to it is the ideal $I_{\Delta}$ generated by the monomials $\prod_{i\in W}x_i$, where $W$ is a non-face of $\Delta$. The Stanley-Reisner ideal associated to the independence complex of a graph $G$, $I_{\Delta(G)}$, is just the edge ideal $I(G)$.

\subsection{Homology}
Our main tool for computing Betti numbers of edge ideals is Hochster's formula \cite{Hoc77}, a classical result describing the Betti numbers of a squarefree monomial ideal in terms of the reduced homology groups of certain simplicial complexes. In the case of edge ideals, this formula can be rewritten as follows:

\begin{prop}[\hspace{-0.0001mm}\cite{RVT07}, Proposition 1.2]\label{prop:HochsterRothVanTuyl} Let $G$ be a simple graph with edge ideal $I(G)$. Then, $$\beta_{i,i+j}(I(G))=\sum_{\substack{W\leq G\\|V(W)|=i+j} }\dim_{\mathbb K} \widetilde{H}_{j-2}(\Delta(W))\ \ for\ all\ i,j\geq0\,.$$
\end{prop}

In particular, $\beta_{i,i+j}(I(G))=0$ whenever $i+j>|V(G)|$. 

In general, computing the dimensions of reduced homology groups is a non-trivial task. Nevertheless, under certain conditions, one can reduce this computation to the one of a somehow simpler simplicial complex. A result in this line is the following (see, e.g.,  \cite[Lemma 2.1]{FRG14} and \cite[Lemma 3.2]{Eng}): \begin{lem}
\label{lema:BasicsHomology}
    Let $G$ be a simple graph.
    \begin{itemize}
        \item[(a)] If $G$ has an isolated vertex, then $\Delta(G)$ is acyclic, i.e., $\widetilde{H}_j(\Delta(G))=0$ for all $j\geq -1$.
        \item[(b)] If $G$ has two vertices $u$ and $v$ such that $N(u)=\{v\}$ and $N(v)=\{u\}$, then $\widetilde{H}_j(\Delta(G))\cong \widetilde{H}_{j-1}(\Delta(G\setminus\{u,v\}))$ for all $j \geq 0$. 
        \item[(c)] If $u,v\in V(G)$ satisfy $N(u)\subseteq N(v)$, then $\widetilde{H}_j(\Delta(G))\cong \widetilde{H}_j(\Delta(G\setminus\{v\}))$ for all $j\geq -1$. 
    \end{itemize}
\end{lem}

We now present some direct applications of the previous result. The first one concerns vertices of degree $1$.

\begin{lem}\label{lem:homologiagrado1}
        Let $G$ be a graph. If $u,v\in V(G)$ satisfy $N(u) = \{v\}$, then $\widetilde{H}_j(\Delta(G))\cong \widetilde{H}_{j-1}(\Delta(G\setminus N[v]))$ for all $j\geq 0$. 
\end{lem}
\begin{proof} Let $T := N(v) \setminus \{u\}$. For every $w \in T$ we have that $N(u) \subseteq N(w)$. Hence, by Lemma \ref{lema:BasicsHomology}.(c), we have that $\widetilde{H}_j(\Delta(G))\cong \widetilde{H}_j(\Delta(G\setminus\{w\}))$ for all $j\geq -1$. Repeating this argument for all vertices in $T$ we get that $\widetilde{H}_j(\Delta(G))\cong \widetilde{H}_j(\Delta(G\setminus T))$ for all $j\geq -1$. Furthermore, in $G\setminus T$ the edge $\{u,v\}$ is isolated; hence, by Lemma  \ref{lema:BasicsHomology}.(b) and taking into account that $N[v]= T \cup \{u,v\}$, we conclude that 
$\widetilde{H}_j(\Delta(G))\cong \widetilde{H}_{j-1}(\Delta(G \setminus N[v]))$ for all $j \geq 0$.
\end{proof}

Complete bipartite graphs $K_{1,r}$ with $r \geq 1$ are usually called {\em star graphs}. Note that star graphs coincide with triangle-free graphs with a vertex $x\in V(G)$ such that $N[x]=V(G)$. 

\begin{lem}\label{lema:HomologiaStarGraph}
$\widetilde{H}_j(\Delta(K_{1,r}))\cong\begin{cases}
        \mathbb K & if \ j=0,\  and\\
        0 & otherwise.
    \end{cases}$
\end{lem}

\begin{lem}\label{lema:HomologiaPalitos}
    $\widetilde{H}_j(\Delta(m\cdot K_2))\cong\begin{cases}
        \mathbb K &  if \ j=m-1,\ and\\
        0 & otherwise.
    \end{cases}$
\end{lem}

Another very useful homological result is Mayer-Vietoris sequence, which was expressed in its modern form in \cite{ES52}.  
A special case of this sequence is the following (see, e.g., \cite{FRG14}):

\begin{thm}\label{thm:MayerVietoris}
    Let $\Delta$ be a simplicial complex, and let $v\in V(\Delta)$ such that $\lk_{\Delta}(v)\neq\{\emptyset\}$. Then there is a long exact sequence \begin{eqnarray*}
    \cdots\rightarrow\widetilde{H}_j(\lk_{\Delta}(v))\rightarrow\widetilde{H}_j(\del_{\Delta}(v))\rightarrow\widetilde{H}_j(\Delta)\rightarrow\widetilde{H}_{j-1}(\lk_{\Delta}(v))\rightarrow\cdots\\
    \cdots \rightarrow \widetilde{H}_0(\lk_{\Delta}(v))\rightarrow\widetilde{H}_0(\del_{\Delta}(v))\rightarrow\widetilde{H}_0(\Delta)\rightarrow 0\,.
   \end{eqnarray*}
\end{thm}

A direct consequence of this Mayer-Vietoris sequence is the following:

\begin{cor}\label{cor:TruquitoMayerVietoris}
Let $\Delta$ be a simplicial complex, $v\in V(\Delta)$ such that ${\rm link}_{\Delta}(v)\neq\{\emptyset\}$ and $r\in\mathbb N$. Then:
\begin{itemize}
    \item If $\widetilde{H}_j({\rm link}_{\Delta}(v))=0$ for all $j\geq r$, then \[\widetilde{H}_j(\Delta)\cong\widetilde{H}_j({\rm del}_{\Delta}(v)) \ \text{for all }  j\geq r+1\,.\]
    \item If $\widetilde{H}_j({\rm del}_{\Delta}(v))=0$ for all $j\geq r$, then \[\widetilde{H}_j(\Delta)\cong\widetilde{H}_{j-1}({\rm link}_{\Delta}(v)) \ \text{for all }  j\geq r+1\,.\]
\end{itemize}
\end{cor}

Now we have all the necessary ingredients to study the edge ideals associated to $\GAsin$. First, note that the graphs  ${\rm GA}(t,2)'$ consist of isolated vertices. Also note that ${\rm GA}(t,3)' = (t+1) \cdot K_2$, and hence its Betti diagram is diagonal and $\beta_{i,2i+2} = \binom{t+1}{i+1}$ for $0 \leq i \leq t$ (see Figure \ref{fig:bettik3}). \begin{figure}
\[
\begin{array}{c|ccccccc}
& 0 &   \cdots & i & \cdots  & t \\
\hline 
2 & t+1 \\
\vdots & & \ddots \\
i+2 & & &  \binom{t+1}{i+1}\\
\vdots & & & & \ddots\\
t+2 & & & & &  1
 \end{array}
\]
\caption{Betti table of $I({\rm GA}(t,3)') = I((t+1) \cdot K_2)$} \label{fig:bettik3} 
\end{figure}
Whenever $k \geq 4$, $\GAsin$ is a connected $(k-2)$-regular graph.

\section{Linear strand}

In 2007, Roth and Van Tuyl obtained the following combinatorial formula for the Betti numbers in the linear strand of  edge ideals.

\begin{prop}[\hspace{-0.0001mm}\cite{RVT07}, Proposition 2.1]\label{prop:LinearStrand} Let $G$ be a simple graph. Then \[
\beta_{i,i+2}(I(G))=\sum_{\substack{W\leq G\\ |V(W)|=i+2}}(\#{\rm comp}(W^c)-1) \ for\ all\ i\geq 0\,,
\] where $\#{\rm comp}(W^c)$ denotes the number of connected components of $W^c$. 
\end{prop}

In this section, we will rewrite this formula for the family of triangle-free graphs and use it for computing the Betti numbers in the linear strand of $I(\GAsin)$.

\begin{lem}\label{lema:LinearStrandTriangleFree}
    Let $G$ be a graph without triangles. Then, for all $i\geq 0$ \[\beta_{i,i+2}(I(G))=\sum_{\substack{a+b=i+2\\ 0<a\leq b}}\#\{W\leq G\,|\, W \ is\ isomorphic \ to \ K_{a,b}\}\, .\]
\end{lem}
\begin{proof}
    Taking into account Proposition \ref{prop:LinearStrand}, let us analyze the possible values of $\#{\rm comp}(W^c)$, where $W$ is an arbitrary induced subgraph of $G$ of size $i+2$.

    Assume first that there exists such an induced subgraph $W\leq G$ with $\#{\rm comp}(W^c)>2$. In this case, there exist at least three connected components $\mathcal C_1$, $\mathcal C_2$ and $\mathcal C_3$ in $W^c$, and a vertex $x_i$ in each $\mathcal C_i$ for $i\in\{1,2,3\}$. Then, the induced subgraph of ${W^c}$ on ${\{x_1,x_2,x_3\}}$ consists of three isolated vertices, and hence the induced subgraph of $W$ on this set is a triangle, which contradicts the assumption that $G$ is triangle-free. Consequently \[\beta_{i,i+2}(I(G))=\sum_{\substack{W\leq G\\ |V(W)|=i+2}}(\#{\rm comp}(W^c)-1)=\#\{ W\leq G \,|\,|V(W)|=i+2\ {\rm and}\ \#{\rm comp}(W^c)=2\}\,.\]

    Furthermore, when $W^c$ has exactly two connected components $\mathcal{C}_1$ and $\mathcal{C}_2$, both of them have to be cliques. Otherwise, if $\mathcal{C}_1$ is not a clique, there exist two vertices $x,y\in\mathcal{C}_1$ not connected by an edge, and these two vertices, together with a vertex in $\mathcal{C}_2$, constitute an independent set of three vertices in $G^c$ that leads to a triangle in $G$, reaching again a contradiction.

    To finish, just notice that an induced subgraph of $G^c$ with two complete connected components leads to an induced subgraph of $G$ which is complete bipartite, and vice versa.
\end{proof}

Let us denote by $k_{a,b}(G)$ (or simply by $k_{a,b}$ when $G$ is clear) the quantity $k_{a,b}:=\#\{W\leq G\ | \ W \cong  K_{a,b}\}$ appearing in Lemma \ref{lema:LinearStrandTriangleFree}. In the following lemma we provide formulas for $k_{a,b}(G)$ when $G$ is a Generalized Andrásfai graph from which the exterior cycle has been removed (see Figure \ref{fig:BipartitoCompletoAnd(6)} for an example of a complete bipartite induced subgraph of ${\rm And}(6)'$). From this lemma we derive in  Proposition~\ref{prop:LinearStrandAndrasfaiSinCiclo} explicit formulas for the values $\beta_{i,i+2}(I(\GAsin))$, the Betti numbers in the linear strand of the edge ideals of these graphs. 
     \begin{figure}[!ht]
 \centering
 	\includegraphics[width=6cm]{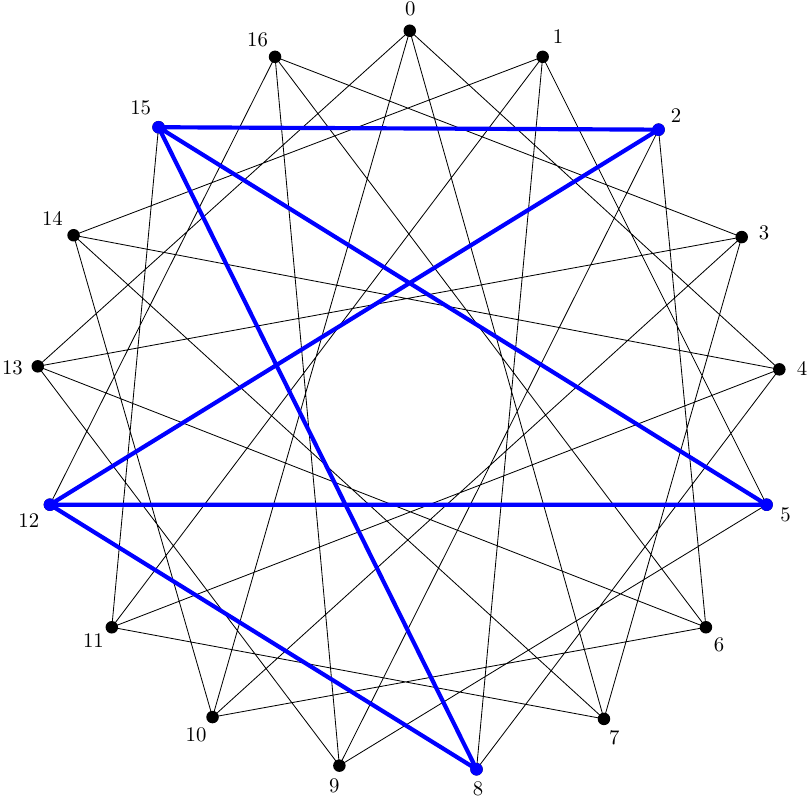}
 	 	\caption{A $K_{2,3}$ induced subgraph of ${\rm And}(6)'$} \label{fig:BipartitoCompletoAnd(6)}
 \end{figure}

\begin{lem}\label{lema:BipartitosAndrasfaiSinCiclo}
Let $t,k \geq 3$ and $a,b\in \mathbb N$ with $0<a\leq b$. Then, \[k_{a,b}(\GAsin)=\begin{cases}
    n\, {k-2\choose a+b-1} &  if\ \  a\neq b , \ and \\
    \frac{n}{2}\, {k-2\choose a+b-1} & if\ \  a= b .
\end{cases}
\]    
\end{lem}
\begin{proof}

    Let us consider $a,b\in\mathbb N$ with $0<a\leq b$ and compute the number of induced subgraphs of $\GAsin$ isomorphic to $K_{a,b}$.

    For every $x \in V(\GAsin)$ and every sequence of integers $1\leq k_2<k_3<\cdots<k_{a}<c_1<c_2<\cdots<c_b \leq k-2$, the induced subgraph on $S = \{x_1,\ldots,x_a,y_1,\ldots,y_b\}$, where $x_1 = x,\, x_i = x + k_i t$ for $i = 2,\ldots,a$, and $y_i = x + c_i t + 1$ for $i = 1,\ldots,b$, is isomorphic to $K_{a,b}$ (see \eqref{eq:GAsin}). 
    
    Take now $S \subseteq \mathbb Z_{n}$ such that the corresponding induced subgraph of $\GAsin$ is isomorphic to $K_{a,b}$, and let us prove that it is of this form. 
    We denote by $A = \{x_1,\ldots,x_a\}$ and $B = \{y_1,\ldots,y_b\}$ the bipartition sets of $S$ and assume, without loss of generality, that $0 = x_1 < x_2 < \cdots < x_a < t(k-1) + 2$, $y_1 < \cdots < y_b < t(k-1) + 2$. Since $y_j \in N(0)$, we have that $y_j \equiv 1\pmod t$  for all $j\in\{1,\dots,b\}$. 
    If $x_j<y_1$ for some $j\in\{2,\dots,a\}$, then the fact that $x_j$ and $y_1$ are neighbors implies that $x_j\equiv 0\pmod t$. On the other hand, if $x_j>y_1$, then $x_j\equiv 2\pmod t$.

    Let $i\in\{1,\dots,a\}$ be the largest index such that $x_i<y_1$. We claim that $x_i<y_1<y_2<\dots<y_b<x_{i+1}$ (see Figure \ref{fig:EsquemaBipartitosInducidosGA}). Otherwise, there exists $j$ such that $y_1<x_j<y_b$; however, this implies that $x_j\equiv2\pmod t$ (because $y_1 < x_j$ and $\{y_1,x_j\}$ is an edge) and $x_j\equiv 0\pmod t$ (because $x_j < y_b$ and $\{x_j,y_b\}$ is an edge), a contradiction. 

     \begin{figure}[!ht]
 \centering
 	\includegraphics[width=5cm]{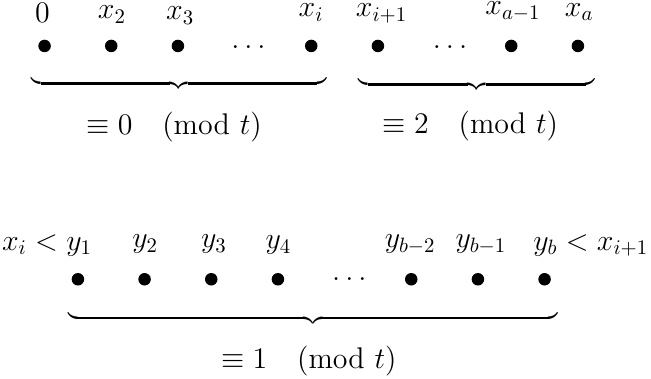}
 	 	\caption{Vertices of the complete bipartite induced subgraph of $\GAsin$} \label{fig:EsquemaBipartitosInducidosGA}
 \end{figure}

By considering $S' = S - x_{i+1} = \{s - x_{i+1} \, \vert \, s \in S\}$, we have that $S'$ also induces a complete bipartite graph with bipartition sets $A' = \{x_1',x_2',\ldots,x_a'\}$ and $ B' = \{y_1',\ldots,y_b'\}$, where $0 = x_1' < x_2' < \cdots < x_a' < y_1' < \cdots < y_b' < t(k-1) + 2,\, x_j' \equiv 0 \ ({\rm mod}\ t)$, and $y_j' \equiv 1 \ ({\rm mod}\ t)$ for all $j$.

 Now we write $x_i'=t k_{i}$ for all $i\in\{2,\dots,a\}$ and $y_j'=t c_j+1$ for all $j\in\{1,\dots,b\}$. Since $\{x_a, y_1\}$ and $\{0,y_b\}$ are edges, we deduce that $c_1 > k_a$ and $c_b \neq k-1$.

Hence, we have seen that there is a correspondence between the $(a+b)$-tuples $(x,k_2,\ldots,k_a,c_1,\ldots,c_b)$, with $x \in V(\GAsin)$ and $1 \leq k_2 < \cdots < k_a < c_1 < \cdots < c_b \leq k-2$, and the induced subgraphs of $\GAsin$ isomorphic to $K_{a,b}$. To derive the formula for $k_{a,b}$, it suffices to observe that this correspondence is bijective whenever $a < b$, and all fibers have size 2 if $a = b$. \end{proof}

\begin{prop}\label{prop:LinearStrandAndrasfaiSinCiclo}
    For all $i\geq0$, \[\beta_{i,i+2}(I(\GAsin))=n\, {k-2 \choose i+1} \left(\frac{i+1}{2}\right)\, .\]
\end{prop}
\begin{proof}
    By Lemma \ref{lema:LinearStrandTriangleFree}, it is enough to compute $\sum_{\substack{a+b=i+2 \\ 0<a\leq b}}k_{a,b}$. The result follows by substituting $k_{a,b}$ by the values obtained in Lemma \ref{lema:BipartitosAndrasfaiSinCiclo}.
    \end{proof}

By Proposition \ref{prop:LinearStrandAndrasfaiSinCiclo}, it is clear that the last nonzero entry in the linear strand of $I(\GAsin)$ is $\beta_{k-3,\,k-1}$.

\begin{cor}\label{cor:SimetriaLinearStrand}
    The linear strand of $I(\GAsin)$ is symmetric, i.e., $$\beta_{i,\,i+2}(I(\GAsin))=\beta_{k-3-i,\,k-1-i}(I(\GAsin))\ \text{ for\ all }\ i\in\{0,1,\dots,k-3\}\,.$$
\end{cor}

As another consequence of Proposition \ref{prop:LinearStrandAndrasfaiSinCiclo}, we obtain the vertex connectivity of the complement of $\GAsin$.

By Proposition \ref{prop:LinearStrand}, we have that $\beta_{i,i+2}(I(G))=0$ if and only if $W^c$ is connected for all $W\leq G$ with $|V(W)|=i+2$, which is  equivalent to $\kappa(G^c) > |V(G)| - (i+2)$.  In terms of the first row of the Betti diagram of $I(G)$, this means that its last nonzero entry belongs to the column $|V(G)|-\kappa(G^c)-2$. This information, together with the fact that the last nonzero element in the linear strand of $I(\GAsin)$ is $\beta_{k-3,k-1}$, yields the following:

\begin{cor}\label{cor:ConectividadSinCiclo}
    The vertex connectivity of $(\GAsin)^c$ is equal to $(t-1)(k-1)+2$.
\end{cor}

Since $(\GAsin)^c={\rm Cay}(\mathbb Z_{n},\,\{m\in\{2,3,\dots,(k-1) t\}\ |\ m\not\equiv1\pmod t\}\cup\{1,\,n-1\})$ and the connection set has cardinality $n-k+1=(t-1)(k-1)+2=\kappa((\GAsin)^c)$, we conclude that $(\GAsin)^c$ is $\kappa$-optimal.

\section{Main diagonal of the Betti diagram}\label{sec:MainDiagonal}

In this section, we will study the main diagonal of the Betti diagram of $I(\GAsin)$. For this purpose, we will use the following result by Katzman:

\begin{lem}[\hspace{-0.0001mm}\cite{Kat06}, Lemma 2.2] \label{lema:MainDiagonalKatzman}
    For any graph $G$ and any $i\geq 0$, $\beta_{i,d}(I(G))=0$ for all $d> 2\,(i+1)$ and $\beta_{i,\,2(i+1)}(I(G))$ is equal to the number of induced matchings of $G$ of size $i+1$.
\end{lem}

Throughout this section, when $\{i,j\} \in E(\GAsin)$ with $0 \leq i < j < n$, we refer to $i$ as the {\em beginning} of the edge, and to $j$ as the {\em end} of the edge.

When $k = 3$, then ${\rm GA}(t,3)' = (t+1) \cdot K_2$ and its induced matching number is $t+1$. For $k \geq 4$, we prove in the following lemma that the induced matching number is $t$, and we compute the number of induced matchings of size $t$ in $\GAsin$ (see Figure \ref{fig:3PalitosAnd(6)} for an example of an induced matching of size 3 in ${\rm And}(6)'$). From this lemma we derive in  Proposition~\ref{prop:DiagonalSinCiclo} an explicit formula for the value $\beta_{t-1,2t}(I(\GAsin))$, the last nonzero entry in the main diagonal of the Betti diagrams of these edge ideals.

 \begin{figure}[!ht]
 \centering
 	\includegraphics[width=6cm]{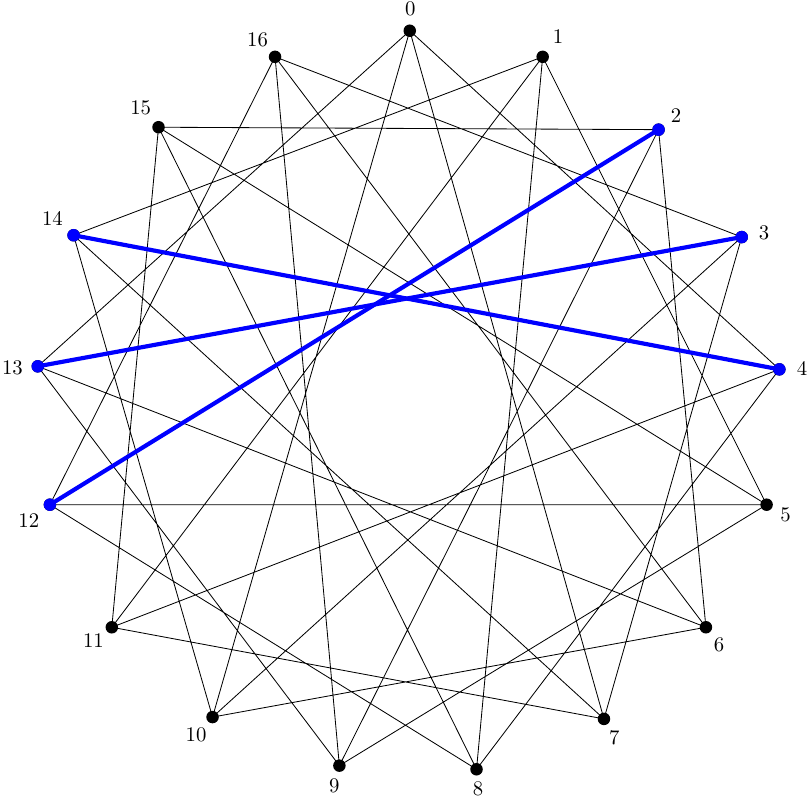}
 	 	\caption{Induced matching of size 3 in ${\rm And}(6)'$} \label{fig:3PalitosAnd(6)}
 \end{figure}

\begin{lem}\label{lema:DiagonalSinCiclo}
    For $k \geq 4$, the induced matching number of $\GAsin$ is $t$, and the number of induced matchings of size $t$ is exactly $\frac{n}{2}\,(t(k-3)+1)$.
\end{lem}
\begin{proof}
For every $x\in V(\GAsin)$ and every $y\in\{x+(t+1),\dots,x+(n-t-1)\}$, the induced subgraph of $\GAsin$ with vertices  $\{x,\,x+1,\dots,\,x+(t-1),\,y,\,y+1,\dots,\,y+(t-1)\}$ is an induced matching of size $t$.

Take now $S\subseteq \mathbb Z_n$ such that the corresponding induced subgraph of $\GAsin$ is an induced matching of maximum size, and let us prove that it is of this form. We denote the edges of this matching by $\{u_1,v_1\},\ldots, \{u_s,v_s\}$, where $s\geq t$, $u_i < v_i < t(k-1)+2$ and $u_1 < u_2 < \cdots < u_s$ and assume, without loss of generality, that $u_1 = 0$ and $v_i \neq t(k-1)+1$ for all $i \in \{1,\ldots,s\}$ (we may assume this since $k \geq 4$ and the whole graph is not an induced matching). 

Let us see first that $s = t$. If $s > t$, there are two edges whose beginnings are congruent modulo $t$. Then $u_i \equiv u_j \  ({\rm mod}\ t)$ and if $1 < i < j$, then $u_i < u_j < v_j$ and $v_j - u_j \equiv v_j - u_i \equiv 1\ ({\rm mod}\ t)$. Since $1 < v_j - u_i < t(k-1) + 1$ (because $v_j \neq t(k-1)+1$), then $\{u_i,v_j\} \in E(\GAsin)$, a contradiction (see Figure \ref{fig:t palitos no misma cong}).

\begin{figure}[!ht]
\centering
\includegraphics[width=6cm]{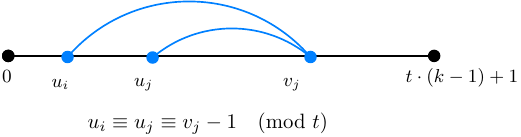}
\caption{Two beginnings that are congruent modulo $t$}  \label{fig:t palitos no misma cong}
\end{figure}

 The previous argument implies that $s = t$ and all possible congruences modulo $t$ will appear among the beginnings of the edges. Furthermore, these beginnings must constitute an independent set.
Since the beginning of one of the $t$ edges has to be congruent to 1 modulo $t$, and it cannot be connected to 0, the only possibility is that this beginning is vertex 1 and then $u_2 = 1$. If we do the same reasoning with the beginning which is congruent to 2 modulo $t$, we get that it has to be vertex 2, and then $u_3 = 2$. Iterating this process, what we get is that $u_i = i-1$ for $i = 1,\ldots,s$. A symmetric argument allows to conclude that the ends of the edges have to be consecutive elements. 

Hence, the set of the beginnings of the edges is $\{0,1,\dots,t-1\}$, and $v_1,\ldots,v_t$ are consecutive vertices. We observe that ${\rm min}\{v_j\} \neq t$, since $t$ it is not connected to any vertex in $\{0,1,\dots,t-1\}$, and this means that ${\rm min}\{v_j\} \geq t+1$. In addition, as ${\rm max}\{v_j\} \leq t (k-1)$, it follows that ${\rm min}\{v_j\} = {\rm max}\{v_j\} - t + 1 \leq t(k-1) -t+1=t(k-2)+1$. 

As a consequence, we have seen that there is a correspondence between the pairs $(x,y)$, with $x\in V(\GAsin)$ and $y\in\{x+(t+1),\dots,x+(n-t-1)\}$ (a set of cardinality $n - 2t - 1 = t(k-3)+1$), and the induced matchings of $\GAsin$ of size $t$. To deduce the total number of induced matchings of $\GAsin$ of size $t$, it suffices to observe that all fibers of this correspondence have size 2, i.e., the pairs $(x,y)$ and $(x',y')$ induce the same subgraph if and only if $\{x,y\} = \{x',y'\}$ . 
\end{proof}

  From Lemmas \ref{lema:MainDiagonalKatzman} and \ref{lema:DiagonalSinCiclo} we directly deduce the following:

\begin{prop}\label{prop:DiagonalSinCiclo}Let $k \geq 4$. Then 
   \[ \beta_{i,\,2(i+1)}(I(\GAsin)) = 0 \text{ if and only if } i \geq t. \] Moreover, $\beta_{t-1,\,2t}(I(\GAsin)) = \frac{n}{2}\,(t(k-3)+1)$.
\end{prop}

In addition, as a consequence of Proposition \ref{prop:DiagonalSinCiclo}, we get a lower bound for the regularity of $I(\GAsin)$. In particular, we deduce that ${\rm reg}(I(\GAsin))\geq t+1$ whenever $k \geq 4$. In the next section we will see that the actual value of the regularity is $t+2$. % Even more, we will see that ${\rm reg}(I(\GAsin))$ is only attained at the last step of the resolution and that $\beta_{p,p+t+2}(I(\GAsin)) = 1$ for $p = {\rm pd}(I(\GAsin))$.

\section{Regularity}\label{sec:Regularity}

In this section, we use homological tools to compute the last nonzero row of the Betti diagram of $I(\GAsin)$ and conclude that its regularity equals $t+2$. More precisely, the main result of this section is the following:
\begin{thm}\label{thm:lastline} Let $t\geq 1$ and $k \geq 3$. For all $i \in \mathbb N$ and all $j \geq t+2$, then
\[ \beta_{i,i+j}(I(\GAsin)) = \left\{ \begin{array}{ll}
    1 &  if \ i+j = n\ and \ j = t+2,\ and \\
    0 & otherwise.
\end{array} \right.\]
In particular, ${\rm reg}(I(\GAsin)) = t+2$.
\end{thm}

This result extends the ones obtained in \cite{FRG09, FRG14}, which correspond to $t = 1$ and $t = 2$, and also generalizes the behavior of ${\rm GA}(t,3)' = (t+1)\cdot K_2$ (see Figure \ref{fig:bettik3}).
To prove it, we first need some auxiliary results.

\begin{lem}\label{lema:Inductivo} Let $t \geq 2$. There exists a graph isomorphism $\phi$ from $\GAsin \setminus (N[0] \cup N[1])$ to  ${\rm GA}(t-1,k)'\setminus N[0]$ such that $\phi(2) = 1$ and $\phi(t(k-1)+1) = (t-1)(k-1)+1$. 
\end{lem}
\begin{proof}
For each $x \in V(\GAsin) \setminus (N[0] \cup N[1])$ with $0 \leq x < t(k-1) + 2$, we write $x = t q + r$, with $q,r \in \mathbb N$ and $0 \leq r < t$, and define \[ \phi(x) = \left\{ \begin{array}{ll} (t-1) q + r & \text{if } r = 0,1, \ \text{and}\\
    (t-1) q + (r-1) & \text{otherwise}.\end{array} \right. \] Then $\phi$ is actually a graph isomorphism satisfying that $\phi(2) = 1$ and $\phi(t(k-1)+1) = (t-1)(k-1)+1$.
 \end{proof}

The previous lemma allows us to prove the following result by induction.

\begin{lem}\label{lema:HomologiaLinkSinCiclo}
Let $t \geq 2$, $H<\GAsin$ and $x \in V(H)$ such that $x-1\notin V(H)$. 
We have the following: \begin{itemize}
\item if $x+1 \notin V(H)$, then $\widetilde{H}_j(\lk_{\Delta(H)}(x))=0$ for all $j\geq t-2$, and
\item if $x+1 \in V(H)$, then $\widetilde{H}_j(\lk_{\Delta(H)}(x))=0$ for all $j\geq t-1$.
\end{itemize}

\end{lem}
\begin{proof}
    We proceed by induction on $t \geq 2$. Assume first that $t = 2$ and consider $H<{\rm GA}(2,k)'=K_{k,k}'$ such that $0\in V(H)$ and $2k-1\notin V(H)$. We distinguish two possibilities for $\lk_{\Delta(H)}(0)$. On the one hand, if $1\notin V(H)$, then it could happen that $H\setminus N[0]$ is empty (and $\widetilde{H}_j(\{\emptyset\})=0$ for all $j\geq 0$) or that $H\setminus N[0]$ is constituted by isolated vertices (which implies that its independence complex is acyclic). On the other hand, if $1\in V(H)$, then $H\setminus N[0]$ is a star graph centered at 1 and, by Lemma \ref{lema:HomologiaStarGraph}, this graph satisfies that $\widetilde{H}_j(\lk_{\Delta(H)}(0))=0$ for all $j\geq 1$.

         Assume now that the result is true for certain $t \geq 2$, and let us prove it for $t+1$.
        Let $H<{\rm GA}(t+1,k)'$, and assume without loss of generality that $x = 0\in V(H)$ and $-1 \notin V(H)$ (note that $(t+1) (k-1) + 1 = -1 \in \mathbb Z_{(t+1)(k-1) + 2})$. 

        Let $H'=H\setminus N[0]$. Since $\Delta(H')=\Delta(H\setminus N[0])=\lk_{\Delta(H)}(0)$,  we want to prove  that:
        \begin{itemize}
        \item if $1 \notin V(H),$ then $\widetilde{H}_j(\Delta(H'))=0$ for all $j \geq t-1$, and
        \item if $1 \in V(H),$ then $\widetilde{H}_j(\Delta(H'))=0$ for all $j \geq t$.
        \end{itemize}

\smallskip

\noindent \underline{Case (1)}: $1 \notin V(H)$. There are two possibilities:

\smallskip \noindent
    \underline{Case (1.1)}: There are no vertices in $V(H)$ congruent to 2 modulo $t+1$. Then $H'\leq{\rm GA}(t+1,k)' \setminus (N[0] \cup N[1] \cup \{-1,\, 2\})$. Hence, by Lemma \ref{lema:Inductivo} we have that $H'$ is isomorphic to an induced subgraph of $\GAsin \setminus (N[0] \cup \{-1,\, 1\})$. In other words, $H'$ is isomorphic to $K \setminus N[0]$, where $K\leq\GAsin\setminus \{-1,\,1\}$.        Consequently, $\widetilde{H}_j(\Delta(H'))\cong\widetilde{H}_j(\lk_{\Delta(K)}(0))$, and by the induction hypothesis this last reduced homology group equals 0 for all $j\geq t-2$. 
        
\smallskip \noindent
      \underline{Case (1.2)}: There are vertices in $V(H)$ congruent to 2 modulo $t+1$. Assume that there at least two such vertices, $y_1 < y_2$. Since in $H'$ there are no vertices congruent to 1 modulo $t+1$, then $N(y_2)\subseteq N(y_1)$ and, by Lemma \ref{lema:BasicsHomology}.(c), we can remove $y_1$ from $H'$ preserving homologies. Thus, we may assume without loss of generality that there is only one vertex $y$ congruent to 2 modulo $t+1$ in $V(H')$. To analyze the homologies of $\Delta(H')$, we are going to study $\del_{\Delta(H')}(y)$ and $\lk_{\Delta(H')}(y)$ and use Mayer-Vietoris sequence.

        By construction, $H'' := H'\setminus\{y\}\leq{\rm GA}(t+1,k)' \setminus (N[0] \cup N[1] \cup \{-1,\, 2\})$. Hence, by Lemma \ref{lema:Inductivo} we have that $H''\cong K\setminus N[0]$, where $K\leq\GAsin\setminus\{-1,\,1\}$. Thus,
        \[\widetilde{H}_j(\del_{\Delta(H')}(y)) \cong \widetilde{H}_j(\Delta(H'\setminus\{y\}))\cong \widetilde{H}_j(\Delta(H''))\cong\widetilde{H}_j(\Delta(K\setminus N[0]))\cong\widetilde{H}_j(\lk_{\Delta(K)}(0)),\] which equals 0 for all $j\geq t-2$ by the induction hypothesis. Hence, by Corollary \ref{cor:TruquitoMayerVietoris}, we deduce that \begin{equation}\label{eq:1} \widetilde{H}_{j+1}(\Delta(H'))\cong \widetilde{H}_j(\lk_{\Delta(H')}(y))\cong\widetilde{H}_j(\Delta(H'\setminus N[y])) \end{equation} for all $j\geq t-2$, where $H'\setminus N[y]\leq{\rm GA}(t+1,k)' \setminus (N[0] \cup N[1] \cup \{-1,\, 2\})$. Again by Lemma \ref{lema:Inductivo},  $H'\setminus N[y]$ is isomorphic to $K' \setminus N[0]$, with $K'\leq\GAsin\setminus\{-1,\,1\}$.  Therefore \begin{equation}\label{eq:2} \widetilde{H}_j(\Delta(H'\setminus N[y]))\cong \widetilde{H}_j(\Delta(K'\setminus N[0]))\cong\widetilde{H}_j(\lk_{\Delta(K')}(0)),\end{equation} where the last reduced homology group is 0 for all $j\geq t-2$ by the induction hypothesis. 

        From \eqref{eq:1} and \eqref{eq:2}, we conclude that $\widetilde{H}_j(\Delta(H'))=0$ for all $j\geq t-1$.

\medskip

\noindent \underline{Case (2)}: $1\in V(H)$. Let $H''=H'\setminus N[1] = H \setminus (N[0] \cup N[1])$. We distinguish two possibilities:

\smallskip \noindent
\underline{Case (2.1)}: If $2\in V(H)$, then by Lemma \ref{lema:Inductivo} $H''\cong K\setminus N[0]$, with $K\leq\GAsin\setminus\{-1\}$ and $1\in V(K)$. Hence,  \[ \widetilde{H}_j(\lk_{\Delta(H')}(1))\cong \widetilde{H}_j(\Delta(H'')) \cong \widetilde{H}_j(\lk_{\Delta(K)}(0)),\] and this is 0 for all $j\geq t-1$ by the induction hypothesis.

    This implies, by Corollary \ref{cor:TruquitoMayerVietoris}, that \[\widetilde{H}_j(\Delta(H'))\cong\widetilde{H}_j(\del_{\Delta(H')}(1))\cong\widetilde{H}_j(\Delta(H'\setminus\{1\}))\ {\rm for \ all\ }  j\geq t,\] where $H'\setminus\{1\}\leq{\rm GA}(t+1,k)' \setminus (N[0]  \cup \{-1,\, 1\})$.

    By \underline{Case (1)}, $\widetilde{H}_j(\Delta(H'\setminus\{1\}))=0$ for all $j\geq t-1$ and this allows us to conclude that $\widetilde{H}_j(\Delta(H'))=0$ for all $j\geq t$.
            
\smallskip \noindent
\underline{Case (2.2)}: If $2\notin V(H)$, proceeding as in \underline{Case (2.1}) we conclude that $\widetilde{H}_j(\Delta(H'))=0$ for all $j\geq t-1$. 
\end{proof}

We now use this result to show that almost all of the entries in the $(t+2)$-th row of the Betti diagram of $I(\GAsin)$ are equal to 0. After this, we will see that there is a nonzero entry in this row, and putting everything together we will conclude the value of the regularity of $I(\GAsin)$.

\begin{prop}\label{prop:UltimaFilaPropio}
    $\beta_{i,i+j}(I(\GAsin))=0$ for all $i+j \leq n-1$ and all $j\geq t+2$.
\end{prop}
\begin{proof}
    By Proposition \ref{prop:HochsterRothVanTuyl}, it suffices to prove that $\dim_{\mathbb K}\widetilde{H}_j(\Delta(W))=0$ for every $j\geq t$ and every $W<\GAsin$.
    
    We proceed by induction on the number of vertices of $W$. If $|V(W)| = 1$, then the result clearly follows.
    Assume $|V(W)| > 1$, and take $x\in V(W)$ such that $x-1\notin V(W)$ (such a vertex exists because $W$ is a proper induced subgraph). 
    If $V(W) = N[x],$ then $W$ is a star graph and the result follows by Lemma \ref{lema:HomologiaStarGraph}. Otherwise, we consider $W\setminus N[x]$. By Lemma \ref{lema:HomologiaLinkSinCiclo}, we have that \[ \widetilde{H}_j(\Delta(W\setminus N[x])) \cong \widetilde{H}_j(\lk_{\Delta(W)}(x)) \cong 0 \ \text{for all } j \geq t-1.\] This implies, by Corollary \ref{cor:TruquitoMayerVietoris}, that $\widetilde{H}_j(\Delta(W))\cong\widetilde{H}_j(\del_{\Delta(W)}(x))$ for all $j\geq t$. Since $\del_{\Delta(W)}(x) = \Delta(W\setminus\{ x\})$, then the result follows from the inductive hypothesis.
\end{proof}

By Proposition \ref{prop:HochsterRothVanTuyl}, we know that $\beta_{i,i+j}(I(\GAsin))=0$ whenever $i+j> n$. In Proposition \ref{prop: UltimaDiagonalNoNula} we analyze what happens for $i+j=n$, and in the proof of this result we repeatedly use the following:

\begin{lem}\label{lema:Borrado}
    Let $W<\GAsin$, and let $x$ be a vertex of $W$ such that $x+1\notin V(W)$ but $x-t\in V(W)$. Then, $N_W(x-t)\subseteq N_W(x)$ and $\widetilde{H}_j(\Delta(W))\cong \widetilde{H}_j(\Delta(W\setminus\{x\}))$ for all $j\geq -1$.
\end{lem}
\begin{proof}
    We observe that $N_{\GAsin}(x-t) \setminus N_{\GAsin}(x) = \{x+1\}$. Hence, if $x,\, x-t \in V(W)$ and $x+1\notin V(W)$, we derive that $N_W(x-t)\subseteq N_W(x)$. By Lemma \ref{lema:BasicsHomology}.(c) we conclude that $\widetilde{H}_j(\Delta(W))\cong \widetilde{H}_j(\Delta(W\setminus\{x\}))$ for all $j\geq -1$.
    \end{proof}

\begin{prop}\label{prop: UltimaDiagonalNoNula}
$\beta_{i,n}(I(\GAsin))=\begin{cases}
    1 & if\ i=n-t-2,\ and\\
    0 & otherwise.
\end{cases}$
\end{prop}
\begin{proof}
    By Proposition \ref{prop:HochsterRothVanTuyl}, we have to prove that \[\dim_{\mathbb K}\widetilde{H}_j(\Delta(\GAsin))=\begin{cases}
        1\ {\rm if} \ j=t,\ {\rm and}\\
        0 \ {\rm otherwise}.
    \end{cases}\]

    For every $v\in V(\GAsin)$, we claim that $\del_{\Delta(\GAsin)}(v)=\Delta(\GAsin\setminus\{v\})$ is acyclic. Indeed, by Lemma \ref{lema:Borrado}, $N_{\GAsin\setminus\{v\}}(v-t-1)\subseteq N_{\GAsin\setminus\{v\}}(v-1)$ and, by  Lemma \ref{lema:BasicsHomology}.(c), we have that 
    \[\widetilde{H}_j(\Delta(\GAsin\setminus\{v\}))\cong\widetilde{H}_j(\Delta(\GAsin\setminus\{v,v-1\})) \,\text{ for all } j \geq -1.\]  Iterating this idea, we derive that
    \[\widetilde{H}_j(\del_{\Delta(\GAsin)}(v))\cong \widetilde{H}_j(\Delta(W))\, \text{ for all } j \geq -1,\] where
    $W$ is the induced subgraph with vertices $V(W) = \{v+1,\,v+2,\dots,\,v+1+t\}$. Since $W$ is an independent set, by Lemma \ref{lema:BasicsHomology}.(a) we deduce that $\del_{\Delta(\GAsin)}(v)$ is acyclic.

    Hence, by Corollary \ref{cor:TruquitoMayerVietoris}, \[ \widetilde{H}_j(\Delta(\GAsin))\cong\widetilde{H}_{j-1}(\lk_{\Delta(\GAsin)}(v))\, \text{ for all } j\geq 0.\]

    Assume now without loss of generality that $v=0$ and denote $G_0 := \GAsin\setminus N[0]$, so \[\widetilde{H}_j(\Delta(\GAsin))\cong\widetilde{H}_{j-1}(\Delta(G_0)) \text{ for all } j\geq 0.\]

    One has that $N_{G_0}(n-t)=\{1\}$, so denoting $G_1 := G_0 \setminus N[1] = \GAsin\setminus (N[0] \cup N[1])$,   
    by Lemma~\ref{lem:homologiagrado1} we have that 
    \[ \widetilde{H}_{j}(\Delta(G_0)) \cong \widetilde{H}_{j-1}(\Delta(G_1)) \text{ for all } j \geq 0.\] 

    If we iterate this process for all $i \in \{1,\ldots,t-2\}$ denoting $G_{i+1}:=G_{i}\setminus N[i+1]=\GAsin\setminus(N[0]\cup\cdots\cup N[i+1])$, then $N_{G_i}(n-t+i)=\{i+1\}$ and  
    by Lemma \ref{lem:homologiagrado1} we have that 
    \[ \widetilde{H}_{j}(\Delta(G_{i})) \cong \widetilde{H}_{j-1}(\Delta(G_{i+1})) \text{ for all } j \geq 0. \]

Finally, we have that $G_{t-1}$ is the induced subgraph on the vertices $\{t,n-1\}$, i.e., $G_{t-1} \cong K_{2}$. Consequently, by Lemma \ref{lema:HomologiaPalitos}, \[\widetilde{H}_j(\Delta(\GAsin))\cong\widetilde{H}_{j-t}(\Delta(K_{2}))\cong\begin{cases}
          \mathbb K & {\rm if}\ j=t,\ {\rm and}\\
          0 & {\rm otherwise} .\end{cases}\]\end{proof}
    
Now we can easily prove the main result of this section.
\medskip

\noindent {\it Proof of Theorem \ref{thm:lastline}.}  The result directly follows from Propositions \ref{prop:UltimaFilaPropio} and \ref{prop: UltimaDiagonalNoNula}.\hfill\qedsymbol

\medskip

In the next section we study the rest of the Betti diagram of $I(\GAsin)$ and, in particular, we determine its projective dimension.

\section{Projective dimension} \label{sec:ProjectiveDimension}

In this section, we prove in Lemma \ref{lema:diagn-1} that $\beta_{i,n-1}(I(\GAsin)) = 0$ for all $i \geq 0$; note that these entries belong to the same diagonal of the Betti diagram of $I(\GAsin)$. After that, in Lemma \ref{lema:Quitar2SinCiclo}, we compute the only nonzero entry in the diagonal $i+j=n-2$. Then we use Alexander duality to determine, in Theorem \ref{thm:pd}, the projective dimension of these edge ideals. 

Before proceeding with the main results of this section, we introduce some preliminary ones.

Let $W<\GAsin$ and denote by $0\leq x_1<x_2<\dots<x_m < n$ the vertices of $V(\GAsin) \setminus V(W)$. For every $i\in\{1,2,\dots,m-1\}$, the {\em interval} between $x_i$ and $x_{i+1}$ is just the set of vertices which are larger than $x_i$ and smaller than $x_{i+1}$. In addition, the {\em interval} between $x_m$ and $x_1$ is defined as the union of the set of vertices which are larger than $x_m$ and the set of vertices which are smaller than $x_1$. The {\em length} of an interval $I$, denoted ${\rm length}(I)$, is the number of vertices in it. Furthermore, we say that two intervals are {\em consecutive} if the intervals between them have all length 0.

\begin{ex}
    Consider ${\rm GA}(3,6)'$, which is represented in Figure \ref{fig:GA(3,6)}, and let $W={\rm GA}(3,6)'\setminus\{3,4,7,12\}$.

 The intervals of $W$ are $I_1 = \emptyset$ between 3 and 4, $I_2 = \{5,6\}$ between 4 and 7, $I_3 = \{8,9,10,11\}$ between 7 and 12, and $I_4 = \{13,14,15,16,0,1,2\}$ between 12 and 3. The lengths of this intervals are ${\rm length}(I_1) = 0,\ {\rm length}(I_2) = 2,\ {\rm length}(I_3) = 4$ and ${\rm length}(I_4) = 7$. We say that $I_4$ and $I_2$ are consecutive, since $I_1$ is the only interval between them and it has length 0.
 \end{ex}

The following lemma reduces the problem of describing $\widetilde{H}_i(\Delta(W))$, with $W<\GAsin$, to the case where all intervals have length at most $t$.

\begin{lem}\label{lema:intervaloslargos}Let $W < \GAsin$ with an interval $I = \{u+1,\ldots,u+s\}$ of length $s > t$. Then
\[ \widetilde{H}_j(\Delta(W)) \cong \widetilde{H}_j(\Delta(W \setminus \{u+t+1,\ldots,u+s\})) \text{ for all $j \geq -1$.} \]
\end{lem}
\begin{proof}Since $u+s \in W,\ u+s+1 \notin W$ and $u+s-t \in W$, then by Lemma \ref{lema:Borrado} we have that $N(u+s-t) \subseteq N(u+s)$ and 
\[ \widetilde{H}_j(\Delta(W)) \cong \widetilde{H}_j(\Delta(W \setminus \{u+s\})) \text{ for all } j \geq -1. \]
Iterating this idea with $u+s-1,\ldots,u+t+1$, we get the result. 
\end{proof}

And the following lemma allows to remove intervals with less than $t$ elements having a consecutive interval with at least $t$ elements. 
\begin{lem}\label{lema:BorrarIntervalosPequeños}
    Let $W < \GAsin$ with two consecutive intervals $I_1$ and $I_2$ such that ${\rm length}(I_1) \geq t$ and ${\rm length}(I_2) < t$. Then,
    
    \[ \widetilde{H}_j(\Delta(W)) \cong\widetilde{H}_j(\Delta(W \setminus I_2)) \text{ for all $j \geq -1$.} \]
\end{lem}
\begin{proof}
By Lemma \ref{lema:intervaloslargos} we may assume that ${\rm length}(I_1) = t$ and we can write $I_1 = \{0,\ldots,t-1\}$ and $I_2 = \{x+1,\ldots,x+s\}$, with $s = {\rm length}(I_2) < t$. 
Consider $y := x+s\ {\rm mod}\ t$; we have that $y \in I_1$ and, since $I_1$ and $I_2$ are consecutive intervals and $x+s+1 \notin V(W)$, we have that $N(y) \subseteq N(x+s)$. Hence, by Lemma \ref{lema:BasicsHomology}.(c), \[ \widetilde{H}_j(\Delta(W)) \cong \widetilde{H}_j(\Delta(W \setminus \{x+s\})) \text{ for all $j \geq -1$.} \]
Iterating the same argument with $x + s - 1,\ldots,x+1$, the result is proved.
\end{proof}

\begin{prop}\label{prop:lIntervalosLongMayorIgualt}
    Let $W < \GAsin$ with $\ell\geq1$ intervals of length at least $t$. Then \[
    \widetilde{H}_j(\Delta(W))\cong\begin{cases}
    \mathbb K^{\ell-1} & if\  j=t-1 , \ and\\
        0 & otherwise,
    \end{cases}
    \] \noindent where $\mathbb K^0\cong 0$ by convention.
\end{prop}
\begin{proof}
Since $\ell \geq 1$, by Lemma \ref{lema:BorrarIntervalosPequeños} we can assume that there are no intervals of length $\leq t-1$. Also, by Lemma \ref{lema:intervaloslargos}, we can assume there are exactly $\ell$ intervals of length exactly $t$.

We prove the result by induction on $\ell$. If $\ell = 1$, then $W$ is an independent set and the result follows from Lemma \ref{lema:BasicsHomology}.(a).
Take now $\ell + 1 \geq 2$ and denote the intervals of $W$ as $I_i=\{v_{i,0},\dots,v_{i,t-1}\}$ for all $i\in\{1,\dots,\ell+1\}$, where $v_{1,0} < v_{2,0} < \ldots < v_{\ell+1,0}$ and $v_{i,s}$ is the only element in $I_i$ which is congruent to $s$ modulo $t$. For simplicity, we assume that $I_1 = \{0,\ldots,t-1\}.$

We will study the homologies of ${\rm link}_{\Delta(W)}(0)$ and ${\rm del}_{\Delta(W)}(0)$ and then use Mayer-Vietoris sequence to derive the result.

We first study $\widetilde{H}_j({\rm link}_{\Delta(W)}(0))$ taking into account that ${\rm link}_{\Delta(W)}(0)\cong\Delta(W\setminus N[0])$. We observe that the only neighbor of $v_{\ell+1,2}$ in $W \setminus N[0]$ is $1$. By Lemma \ref{lem:homologiagrado1}, we have that \[ \widetilde{H}_{j}(\Delta(W\setminus N[0])) \cong  \widetilde{H}_{j-1}(\Delta(W\setminus (N[0] \cup N[1]))) \text{ for all } j \geq 0. \]

If we iterate this process for $v_{\ell+1,s}$ with $s \in \{3,\ldots,t-1\}$, we get that 
    \[ \widetilde{H}_{j}(\Delta(W \setminus N[0])) \cong \widetilde{H}_{j-(t-2)}(\Delta(W \setminus (N[0] \cup \cdots \cup N[t-2]))) \text{ for all } j \geq 0. \]

Finally, we have that $V(W) \setminus (N[0] \cup \cdots \cup N[t-2])) = \{t-1,v_{2,0},\ldots,v_{\ell+1,0}\}$  , i.e., $W \setminus (N[0] \cup \cdots \cup N[t-2])) \cong K_{1,\ell}$. Consequently, by Lemma \ref{lema:HomologiaStarGraph}, \[\widetilde{H}_j({\rm link}_{\Delta(W)}(0))\cong\widetilde{H}_{j-(t-2)}(\Delta(K_{1,\ell}))\cong\begin{cases}
          \mathbb K & {\rm if}\ j=t-2,\ {\rm and}\\
          0 & {\rm otherwise} .\end{cases}\]

We now study $\del_{\Delta(W)}(0)=\Delta(W\setminus\{0\})$. In $W \setminus \{0\}$ there is an interval of length $t-1$ and $\ell\geq 1$ intervals of length $t$. Then, by Lemma \ref{lema:BorrarIntervalosPequeños} we know that $\widetilde{H}_{j}(\del_{\Delta(W)}(0)) \cong \widetilde{H}_{j}(\Delta(W \setminus I_1))$ and, by the inductive hypothesis,  \[
    \widetilde{H}_j(\del_{\Delta(W)}(0))\cong\begin{cases}
    \mathbb K^{\ell-1} & \text{if}\  j=t-1 , \ \text{and}\\
        0 & \text{otherwise}.
    \end{cases}
    \]

By the Mayer-Vietoris sequence, $\widetilde{H}_j(\Delta(W))\cong\widetilde{H}_j(\del_{\Delta(W)}(0)) \cong 0$ for all $j\notin\{t-2,\,t-1\}$. Let us study now the cases $j = t-2$ and $j = t-1$. Mayer-Vietoris sequence can be written as follows: 

\[
\begin{aligned}
\cdots\; \to\; 
\underbrace{\widetilde{H}_t(\Delta(W))}_{ \cong\  0}\
&\to&
\underbrace{\widetilde{H}_{t-1}\!\bigl(\operatorname{link}_{\Delta(W)}(0)\bigr)}_{ \cong\  0} 
&\to&
\underbrace{\widetilde{H}_{t-1}\!\bigl(\operatorname{del}_{\Delta(W)}(0)\bigr)}_{\cong\  \mathbb K^{\ell-1}}
&\to&
\widetilde{H}_{t-1}(\Delta(W))
&\to
\\
&\to&
\underbrace{\widetilde{H}_{t-2}\!\bigl(\operatorname{link}_{\Delta(W)}(0)\bigr)}_{\cong \ \mathbb K}
&\to&
\underbrace{\widetilde{H}_{t-2}\!\bigl(\operatorname{del}_{\Delta(W)}(0)\bigr)}_{\cong \ 0}
&\to&
\widetilde{H}_{t-2}(\Delta(W))
&\to
\\
&\to&
\underbrace{\widetilde{H}_{t-3}\!\bigl(\operatorname{link}_{\Delta(W)}(0)\bigr)}_{\cong \ 0}
&\to&
\underbrace{\widetilde{H}_{t-3}\!\bigl(\operatorname{del}_{\Delta(W)}(0)\bigr)}_{\cong \ 0}
&\to&
\underbrace{\widetilde{H}_{t-3}(\Delta(W))}_{\cong \ 0}
&\to\;   0
\end{aligned}
\]

 \noindent and it allows us to conclude the following:
\begin{itemize}
    \item $\widetilde H_{t-2}(\Delta(W))\cong0$.

    \item There is an exact sequence $0\rightarrow \mathbb K^{\ell-1}\rightarrow\widetilde{H}_{t-1}(\Delta(W))\rightarrow\mathbb K\rightarrow 0$, so 
     $\widetilde{H}_{t-1}(\Delta(W))\cong \mathbb K^{\ell}$.\end{itemize}\end{proof}

Now we can proceed with the main results of this section.

\begin{lem}\label{lema:diagn-1}
    $\beta_{i,\,n-1}(I(\GAsin))=0$ for all $i\geq 0$.
\end{lem}
\begin{proof}
    By Proposition \ref{prop:HochsterRothVanTuyl}, if suffices to show that $\widetilde{H}_j(\Delta(W))=0$ for every $j\geq0$ and every $W < \GAsin$ with $|V(W)|=n-1$. If $W < \GAsin$ has $n-1$ vertices, then it has a unique interval of length $n-1 \geq t$ and the result follows by Proposition \ref{prop:lIntervalosLongMayorIgualt}. \end{proof}

\begin{lem}\label{lema:Quitar2SinCiclo}
    $\beta_{n-t-3,\,n-2}(I(\GAsin))=\frac{n}{2}\,(t\,(k-3)+1)$ and $\beta_{i,\,n-2}(I(\GAsin))=0$ for every $i\neq n-t-3$.
\end{lem}
\begin{proof}
    By Proposition \ref{prop:HochsterRothVanTuyl}, we know that \begin{equation} \label{eq:n-2} \beta_{i,\,n-2}(I(\GAsin))=\sum_{\substack{W < \GAsin\\|V(W)|=n-2}}\dim_{\mathbb K}\widetilde{H}_{n-i-4}(\Delta(W))\, .\end{equation}

    Let $W < \GAsin$ with $|V(W)|=n-2$, which has two intervals. If one of these intervals has length $\leq t-1$, then by Proposition \ref{prop:lIntervalosLongMayorIgualt} $\widetilde{H}_{j}(\Delta(W)) \cong 0$ for all $j \geq -1$. If both intervals have length $\geq t$, again by Proposition \ref{prop:lIntervalosLongMayorIgualt} we have that
    \[\widetilde{H}_j(\Delta(W)) \cong\begin{cases}
        \mathbb K & {\rm if}\ j=t-1, \ {\rm and}\\
        0 & {\rm otherwise}.
    \end{cases}\]

    This expression together with \eqref{eq:n-2} allow us to conclude that $\beta_{i,\,n-2}(I(\GAsin)) = 0$  whenever $n-i-4\neq t-1$ or, equivalently, $i\neq n-t-3$, and
    \begin{eqnarray*}
    \beta_{n-t-3,\,n-2}&=& \#\{\{x_1,\,x_2\}\,|\, {\rm the\ length \ of \ the\ two\ intervals\ between}\ x_1\ {\rm and}\ x_2\ {\rm is} \geq t\}=\\
        & =& \frac{n}{2}\,(t\,(k-3)+1)\,.
       \end{eqnarray*}\end{proof}

If $W <\GAsin$ with $|V(W)|\geq n -k+1$, then $W$ has at least one interval of length $\geq t$. Thus, by Proposition \ref{prop:lIntervalosLongMayorIgualt}, we deduce that $\widetilde{H}_j(\Delta(W))=0$ for every $j\neq t-1$. Consequently, using Hochster's formula we get the following:

 \begin{prop}\label{prop:CerosEncimaDel1}
     $\beta_{i,i+j}(I(\GAsin))=0$ for every $n-k < i+j  < n$ with $j\neq t+1$.
 \end{prop}

The previous result together with Proposition \ref{prop: UltimaDiagonalNoNula} yield the following:

 \begin{prop}\label{prop:PrimeroPdAndrasfaiSin}
     If $t\leq k+1$, then ${\rm pd}(I(\GAsin))=n-t-2$.
 \end{prop}

 In the forthcoming Theorem \ref{thm:pd} we are going to show that we can drop the assumption of $t\leq k+1$ in the previous result. As  $\beta_{n-t-2,\,n} = 1$, we already have that ${\rm pd}(I(\GAsin)) \geq n-t-2$. So to prove Theorem \ref{thm:pd} we just need to prove that ${\rm pd}(I(\GAsin)) \leq n-t-2$ whenever $t > k+1$. To do that we first introduce the concept of Alexander duality and some related results.

\begin{defi}
    Given a simplicial complex $\Delta$, we define the {\rm Alexander dual} of $\Delta$ as the simplicial complex $\Delta^{\vee}$ whose faces are the complements of the non-faces of $\Delta$, i.e. $\Delta^{\vee}=\{F^c\,|\, F\notin\Delta\}$. 
    
    The {\rm Alexander dual} of the Stanley-Reisner ideal $I_{\Delta}$ is the Stanley-Reisner ideal $I_{\Delta}^{\vee}:=I_{\Delta^{\vee}}$.
\end{defi}

Our motivation for using Alexander duality comes from the following result by Bayer, Charalambous and Popescu, which relates the regularity of a Stanley-Reisner ideal to the projective dimension of its Alexander dual.

\begin{prop}[\hspace{-0.0001mm}\cite{BCP99}, Corollary 2.9]\label{prop:RegPdDual}
    Given a simplicial complex $\Delta$, ${\rm reg}(I_{\Delta})={\rm pd}(I_{\Delta^{\vee}})+1$.
\end{prop}

This result, together with the facts that $(\Delta^{\vee})^{\vee}=\Delta$ for every simplicial complex $\Delta$ and the edge ideal of a graph $G$ is just the Stanley-Reisner ideal associated to the independence complex of $G$, tells us that the projective dimension of $I(\GAsin)$ is upper bounded by $t(k-2)$ if and only if ${\rm reg}(I(\GAsin)^{\vee})\leq t (k-2)+1$. To prove this we  need several auxiliary results concerning Betti numbers of the Alexander dual of a Stanley-Reisner ideal and joins of simplicial complexes.  Part (a) of this result is essentially included in \cite{ER98} (see also \cite[Theorem 1.5.23]{Jac04}), and part (b) is \cite[Corollary 2.7]{BCP99}. 

\begin{thm}\label{thm:BettiDual} 
Let $\Delta$ be a simplicial complex on $\{x_1,\dots,x_n\}$. For all $0\leq i\leq n-1$ and $m\geq i+1$, we have that
\begin{itemize}
\item[(a)] $\beta_{i,m}(I_{\Delta}^{\vee})=\sum_{\substack{F\in\Delta\\|F|=n-m}}\dim_{\mathbb K}\widetilde{H}_{i-1}({\rm link}_{\Delta}F)$, and
\item[(b)]  $\beta_{i,m}(I_{\Delta}^{\vee})\leq\sum_{k=0}^{n-m}{m+k\choose k} \beta_{m-i-1,m+k}(I_{\Delta})$.
\end{itemize}
\end{thm}

The following classical result describes the homologies of $\Delta_1*\Delta_2$ in terms of the individual homologies. 

\begin{thm}\label{thm:HomologyJoin} If $\Delta_1$ and $\Delta_2$ are two simplicial complexes, then for all $i$ $$\widetilde{H}_i(\Delta_1*\Delta_2)\cong\bigoplus_{a+b=i-1}\left(\widetilde{H}_a(\Delta_1)\otimes_{\mathbb K}\widetilde{H}_b(\Delta_2)\right)\,.$$ In particular, $\dim_{\mathbb K}(\widetilde{H}_i(\Delta_1*\Delta_2))=\sum_{a+b=i-1}\left(\dim_{\mathbb K}(\widetilde{H}_a(\Delta_1))\dim_{\mathbb K}(\widetilde{H}_b(\Delta_2))\right)$.
\end{thm}

Given two vertices $x,y$ in $\GAsin$ with $0 \leq x < y < n$, when $1 < y-x < t$, then $\GAsin \setminus (N[x] \cup N[y])$ has two connected components. Moreover, these components are isomorphic to certain induced subgraphs of ${\rm GA}(t_1,k)'$ and ${\rm GA}(t_2,k)'$ for some $t_1, t_2 < t$ such that $t_1 + t_2 = t+1$. This fact is exploited in the following result, and will be a key tool for proving Theorem \ref{thm:RegDualAndrasfaiSin}.

\begin{prop}\label{prop:DescomposicionDisjuntaLink}
     Let $G:=\GAsin$, and let $v$ be a nonzero vertex of $G$ with $1<v<t$. Then ${\rm link}_{\Delta(G)}(\{0,v\})\cong{\rm link}_{\Delta({\rm GA}(v,k)')}(\{0,v\})*{\rm link}_{\Delta({\rm GA}(t-v+1,k)')}(\{0,1\}).$
\end{prop}
\begin{proof}
    Since the independence complex of the disjoint union of two graphs is equal to the join of the independence complexes of the two graphs, it is enough to show that $G\setminus(N[0]\cup N[v])$ is the disjoint union of two graphs $G_1$ and $G_2$, with $G_1\cong {\rm GA}(v,k)'\setminus(N[0]\cup N[v])$ and $G_2\cong{\rm GA}(t-v+1,k)'\setminus( N[0]\cup N[1])$.

    Consider $G'=G\setminus(N[0]\cup N[v])$, which has two connected components $G_1$ and $G_2$.

    On the one hand, $V(G_1)$ is constituted by vertex 1, all vertices congruent to $i$ modulo $t$ with $i\in\{2,3,\dots,v-1\}$, and all vertices congruent to $v$ modulo $t$ except $v$. The graph $G_1$ is isomorphic to ${\rm GA}(v,k)'\setminus(N[0]\cup N[v])$ via the map which sends 1 to 1 and every vertex $t j +i$ congruent to $i$ modulo $t$ (with $i\in\{2,3,\dots,v-1,v\}$) to the vertex $v j+i$ congruent to $i$ modulo $v$.

    On the other hand, $V(G_2)$ is constituted by the vertices $v+1$, $-1$, all vertices which are congruent to $i$ modulo $t$ with $i\in\{v+2,v+3,\dots,t-1\}$, and all vertices which are congruent to 0 modulo $t$ except 0. The graph $G_2$ is isomorphic to ${\rm GA}(t-v+1,k)'\setminus( N[0]\cup N[1])$ via the map which sends $v+1$ to $2$, $-1$ to $-1$ and every vertex $t j+i$ congruent to $i$ modulo $t$ (with $i\in\{v+2,v+3,\dots,t\}$) to the vertex $(t-v+1) j+(i-v+1)$ congruent to $i-v+1$ modulo $t-v+1$.
\end{proof}

With these ingredients we are now able to prove the following result, which will allow us to immediately obtain the projective dimension of $I(\GAsin)$:

\begin{thm}\label{thm:RegDualAndrasfaiSin}
    If $t\geq k$, then ${\rm reg}(I(\GAsin)^{\vee})\leq t(k-2)+1$. 
\end{thm}
\begin{proof}

    Throughout this proof, for all $i,j$ we denote $\beta_{i,j}(I(\GAsin))$ by just $\beta_{i,j}$, and $\beta_{i,j}(I(\GAsin)^{\vee})$ by just $\beta_{i,j}^{\vee}$, and the independence complex of $\GAsin$ by just $\Delta$.

    The regularity of $I(\GAsin)^{\vee}$ is upper bounded by $t(k-2)+1$ if and only if $\beta_{i,i+j}^{\vee}$ equals 0 for every $i\geq 0$ and every $j\geq t(k-2)+2$. From Theorem \ref{thm:BettiDual}.(a), this is also equivalent to the fact that for every $F\in\Delta$ with $|F|\in\{0,1,\dots,t\}$, $\widetilde{H}_i({\rm link}_{\Delta}F)=0$ for all $i\in\{-1,0,\dots,t-|F|-1\}$. To prove this, we distinguish several cases:

\smallskip

\noindent \underline{Case (1)}: If $|F|=0$, then ${\rm link}_{\Delta}F=\Delta$ and by Proposition \ref{prop: UltimaDiagonalNoNula} we know that $\widetilde{H}_i(\Delta)=0$ for every $i\neq t$.
    
 \medskip

\noindent \underline{Case (2)}:  If $|F|=1$, then ${\rm link}_{\Delta} F=\Delta(W)$, where $W$ is a proper induced subgraph of $\GAsin$ with $|V(W)|=n-k+1$. Hence by Proposition \ref{prop:lIntervalosLongMayorIgualt} we know that $\widetilde{H}_i({\rm link}_{\Delta} F)\cong\widetilde{H}_i(\Delta(W))\cong0$ for every $i\neq t-1$.

\medskip

\noindent \underline{Case (3)}: If $|F|\in\{2,3,\dots,k-1\}$, let us prove that $\widetilde{H}_i({\rm link}_{\Delta}F)=0$ for every $i\in\{-1,0,\dots,t-|F|-1\}$.

    Since by Theorem \ref{thm:BettiDual}.(a) we have that $\beta_{i,n-|F|}^{\vee}=\sum_{\substack{S\in\Delta\\|S|=|F|}}\dim_{\mathbb K}\widetilde{H}_{i-1}({\rm link}_\Delta S)\,,$ it suffices to show that $\beta_{i,n-|F|}^\vee=0$ for every $i\in\{0,1,\dots,t-|F|\}$, where by Theorem \ref{thm:BettiDual}.(b) $$\beta_{i,n-|F|}^\vee\leq\sum_{a=0}^{|F|}{n-|F|+a\choose a}\, \beta_{n-|F|-i-1,(n-|F|-i-1)+(i+1+a)}\,.$$

    Consequently it is enough to show that the Betti numbers $\beta_{n-|F|-i-1,(n-|F|-i-1)+(i+1+a)}$ are equal to 0 for every $i\in\{0,1,\dots,t-|F|\}$ and every $a\in\{0,\dots,|F|\}.$

    Since $(n-|F|-i-1)+(i+1+a)=n-|F|+a\geq n-|F|> n-k$, Proposition \ref{prop:CerosEncimaDel1} tells us all the previous Betti numbers are 0 except when $i=t-|F|$ and $a=|F|$. However, in this case Proposition \ref{prop: UltimaDiagonalNoNula} allows us to conclude that $\beta_{n-t-1,n}$ is also equal to 0.

\medskip

\noindent \underline{Case (4)}: If $|F|\in\{k,k+1,\dots,t\}$, let us show that $\widetilde{H}_i({\rm link}_{\Delta}F)=0$ for every $i\in\{-1,0,\dots,t-|F|-1\}$.

    We proceed by induction on $t$. If $t=3$, then $k=3$, $|F|=3$ and ${\rm GA}(3,3)'\cong4\cdot K_2$. In this case ${\rm link}_{\Delta({\rm GA}(3,3)')}F$ is isomorphic to $\Delta(K_2)$ or $\Delta(2\cdot K_2)$ and $\widetilde{H}_{-1}({\rm link}_{\Delta({\rm GA}(3,3)')}F)\cong 0$.

   Assume now that the result is true for every positive integer $t'$ with $k\leq t'<t$, and let us prove it for $t$.

        When considering the link of a face $F$ with at least $k$ vertices, there must be two vertices in $F$ such that the length of the interval between them is strictly smaller than $t-1$ (otherwise, the number of vertices would be at least $t k>t(k-1)+2$). Without loss of generality, we can assume that these vertices are 0 and $v$, where $v<t$ is the smallest vertex in $F$ with this property.

        If we focus first on the link of $\{0,v\}$, we can distinguish two new possibilites:

    \smallskip \noindent
\underline{Case (4.1)}: If $v=1$, then by Lemma \ref{lema:Inductivo} we have that $\GAsin\setminus(N[0]\cup N[1])$ is isomorphic to ${\rm GA}(t-1,k)'\setminus N[0]$. Hence ${\rm link}_{\Delta}(\{0,1\})\cong{\rm link}_{\Delta({\rm GA}(t-1,k)')}(\{0\})$ and doing the link of the rest of the vertices we get that ${\rm link}_{\Delta}F ={\rm link}_{\Delta}(\{0,1\}\cup (F\setminus\{0,1\}))\cong{\rm link}_{\Delta({\rm GA}(t-1,k)')}(\{0\}\cup F_1')$,
     where $F_1=\{0\}\cup F_1'$ satisfies that $|F_1|=|F|-1$.

     As a consequence, $\widetilde{H}_i({\rm link}_{\Delta}F)\cong\widetilde{H}_i({\rm link}_{\Delta({\rm GA}(t-1,k)')}F_1)$ and, by the induction hypothesis, this is  0 for every $i\in\{-1,0,\dots,t-1-|F_1|-1\}$, where $t-1-|F_1|-1=t-|F|-1$.

     \smallskip \noindent
\underline{Case (4.2)}: If $v\neq 1$, then by Proposition \ref{prop:DescomposicionDisjuntaLink} we have that ${\rm link}_{\Delta}(\{0,v\})\cong{\rm link}_{\Delta({\rm GA}(v,k)')}(\{0,v\})*{\rm link}_{\Delta({\rm GA}(t-v+1,k)')}(\{0,1\})$. Consequently,
    \begin{eqnarray*}
        {\rm link}_{\Delta}F&=&{\rm link}_{\Delta}(\{0,v\}\cup(F\setminus\{0,v\}))\cong\\
        &\cong&{\rm link}_{\Delta({\rm GA}(v,k)')}(\{0,v\}\cup F_1')*{\rm link}_{\Delta({\rm GA}(t-v+1,k)')}(\{0,1\}\cup F_2')
    \end{eqnarray*}
    
    \noindent for some $F_1'$ and $F_2'$, where $F_1=\{0,v\}\cup F_1'$ and $F_2=\{0,1\}\cup F_2'$ satisfy that $|F_1|+|F_2|=|F|+2$.

    Now we fix $i\in\{-1,0,\dots,t-|F|-1\}$, and by Theorem \ref{thm:HomologyJoin} we have that $$\dim_{\mathbb K} \widetilde{H}_i({\rm link}_{\Delta} F)=\sum_{p+q=i-1}\left(\dim_{\mathbb K}\widetilde{H}_p({\rm link}_{\Delta({\rm GA}(v,k)')}F_1)\, \dim_{\mathbb K}\widetilde{H}_q({\rm link}_{\Delta({\rm GA}(t-v+1,k)')}F_2)\right).$$

   To finish the proof it suffices to see that all the summands in the expression above are equal to 0. Indeed, if $p\leq v-|F_1|-1$, then the first factor $\dim_{\mathbb K}\widetilde{H}_p({\rm link}_{\Delta({\rm GA}(v,k)')}F_1)$ equals 0 by the induction hypothesis. In addition, when $p\geq v-|F_1|$, then 
    \begin{eqnarray*}
        q&=&i-1-p\leq (t-|F|-1)-1-(v-|F_1|)=\\
        &=&t-|F|-2-v+(|F|+2-|F_2|)=\\
        &=&t-v-|F_2|=(t-v+1)-|F_2|-1
    \end{eqnarray*} and the second factor $\dim_{\mathbb K}\widetilde{H}_q({\rm link}_{\Delta({\rm GA}(t-v+1,k)')}F_2)$ is zero by the induction hypothesis.   \end{proof}

Combining Propositions \ref{prop:PrimeroPdAndrasfaiSin} and \ref{prop:RegPdDual} and Theorem \ref{thm:RegDualAndrasfaiSin}, we conclude the following:

\begin{thm}\label{thm:pd}
   ${\rm pd}(I(\GAsin))=n-t-2=t(k-2)$.
\end{thm}

In addition, we have proved that the shape of the Betti diagram of $I(\GAsin)$ is contained in the shaded area of Figure \ref{fig:ShadowBettiDiagram}.

\begin{figure}[!ht]
 \centering
 	\includegraphics[width=14cm]{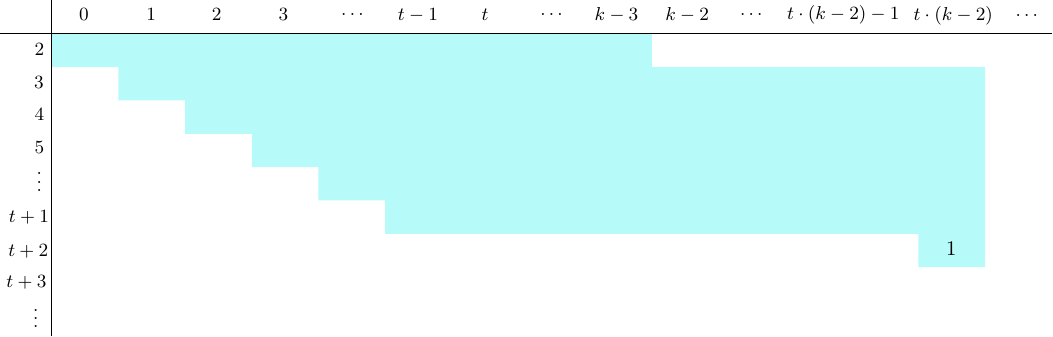}
    \caption{Area containing the exact shape of the Betti diagram of $I(\GAsin)$}
    \label{fig:ShadowBettiDiagram}
 \end{figure}

 Furthermore, we can determine the exact shape of the Betti diagram of $I(\GAsin)$ in the case $t=3$. We devote the next section to this problem.
 
\section{Exact shape of the Betti diagram of $I({\rm And}(k)')$}
 \label{sec:t3}

The goal of this section is to  describe the exact shape of the Betti diagram of $I({\rm And}(k)')$, which is given by the following result (see Figure \ref{fig:shape t=3}).

\begin{figure}[!ht]
 \centering
 	\includegraphics[width=15cm]{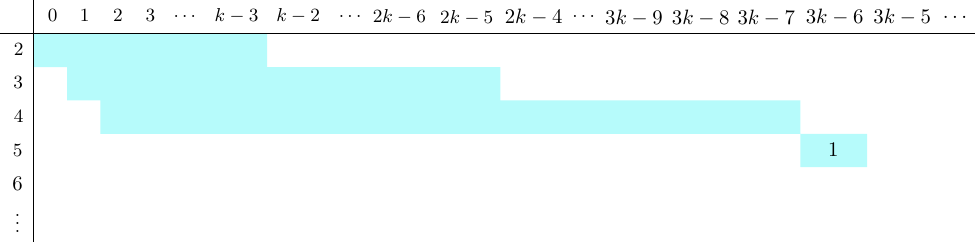}
    \caption{Exact shape of the Betti diagram of $I({\rm And}(k)')$}
    \label{fig:shape t=3}
 \end{figure}

 \begin{thm}\label{thm:t3} Let $k \geq 3$ and consider ${\rm And}(k)'$, the Andrásfai graph from which the exterior cycle has been removed. Then:
\begin{itemize}
\item[(1)] if $\beta_{i,i+j}(I({\rm And}(k)')) \neq 0$, then $2 \leq j \leq 5$,
\item[(2)] $\beta_{i,i+2}(I({\rm And}(k)')) \neq 0$ if and only if $i\in\{0,\dots,k-3\},$
\item[(3)] $\beta_{i,i+3}(I({\rm And}(k)')) \neq 0$ if and only if $i\in\{1,\dots,2k-5\},$
\item[(4)] $\beta_{i,i+4}(I({\rm And}(k)')) \neq 0$ if and only if $i\in\{2,\dots,3k-7\},$ and
\item[(5)] $\beta_{i,i+5}(I({\rm And}(k)')) \neq 0$ if and only if $i = 3k-6.$
\end{itemize}
\end{thm}

Most of the statements in this result are obtained by gathering results from the previous sections. Indeed, (1) is equivalent to ${\rm reg}(I({\rm And}(k)')) \leq 5$, and hence it follows from  Theorem \ref{thm:lastline}. Statement (2) follows from the description of the linear strand provided in Proposition \ref{prop:LinearStrandAndrasfaiSinCiclo}.  Concerning (3), we have that $\beta_{i,i+3}(I({\rm And}(k)')) = 0$ for all $i \geq 2k-3$ by Proposition \ref{prop:CerosEncimaDel1}. In this section we will prove that  $\beta_{2k-4,2k-1} = 0$ (Lemma \ref{lema:UltimoCero3}) and that $\beta_{i,i+3} \neq 0$ for $i\in\{1,\dots,2k-5\}$ (Lemma \ref{lem:Fila3Andrasfai}). 
Concerning (4), we know that $\beta_{i,i+4}(I({\rm And}(k)')) = 0$ for all  $i \geq 3k-5$ by Theorem \ref{thm:pd}, and for $i = 3k-6$ by Lemma \ref{lema:diagn-1}. In this section we will prove that $\beta_{i,i+4} \neq 0$ for $i\in\{2,\dots,3k-7\}$ (Lemma \ref{lem:fila4And}). Finally, (5) follows from Theorem \ref{thm:lastline}.

 \begin{lem}\label{lema:UltimoCero3}
     $\beta_{i,2k-1}(I({\rm And}(k)'))=0$ for every $i\neq 2k-5$.
 \end{lem}
\begin{proof}

    By Proposition \ref{prop:HochsterRothVanTuyl}, it suffices to show that $\widetilde{H}_j(\Delta(W))=0$ for every $j\neq 2$ and every $W<{\rm And}(k)'$ with $|V(W)|=2k-1$, which corresponds to removing $k$ vertices from ${\rm And}(k)'$.

    If $W$ has at least one interval of length at least 3, this is just Proposition \ref{prop:lIntervalosLongMayorIgualt}. Assume now that the length of all the intervals in $W$ is strictly smaller than 3. We necessarily have that $W$ has one interval of length $1$ and $k-1$ intervals of length $2$, and that there are no two consecutive vertices in ${\rm And}(k)'\setminus W$. Then, we may assume that $V(W) = \{ i \in \mathbb Z_{3k-1} \, \vert \, 0 < i < 3k-1$ and $i \not\equiv 0 \ ({\rm mod}\ 3)\}.$ Hence $2$ is an isolated vertex of $W$, and all the reduced homology groups of its independence complex are~$0$. 
\end{proof}

Now we show that there are induced subgraphs of ${\rm And}(k)'$ obtained by removing at least $k+1$ vertices such that the first reduced homology groups of their independence complexes are not 0. This can be written in terms of the Betti numbers as follows: 

\begin{lem}\label{lem:Fila3Andrasfai}
    $\beta_{i,i+3}(I({\rm And}(k)'))\neq 0$ for all $i\in\{1,\dots,2k-5\}$.
\end{lem}
\begin{proof}
    By Proposition \ref{prop:HochsterRothVanTuyl}, for every $i\in\{1,\dots,2k-5\}$ it is enough to find an induced subgraph $W_i$ of ${\rm And}(k)'$ with $|V(W_i)|=i+3$ such that $\widetilde{H}_1(\Delta(W_i))\neq0$.

First, we can consider the induced subgraph $W_{2k-5}$ of ${\rm And}(k)'$ on $\{v_1,\dots,v_{2k-2}\}$, with $0< v_1<\cdots<v_{2k-2}=n-1$ and $\{v_1,\dots,v_{2k-2}\}=(\{v\in V({\rm And}(k)')\,|\, v\equiv i\pmod t,\,i\in\{0,1\}\}\cup\{2\})\setminus\{0,1,4\}$. By repeatedly applying Lemma \ref{lema:Borrado}, we have that $\widetilde{H}_j(\Delta(W_{2k-5}))\cong\widetilde{H}_j(\Delta(2\cdot K_2))$ for all $j\geq -1$.
    
Now, if we define $W_i:=W_{i+1}\setminus\{v_{i+4}\}$ for every $i\in\{1,\dots,2k-6\}$, then $|V(W_i)|=i+3$ and, again by Lemma \ref{lema:Borrado}, $\widetilde{H}_j(\Delta(W_i))\cong\widetilde{H}_j(\Delta(2\cdot K_2))$ for all $j\geq -1$.

    We conclude by Lemma \ref{lema:HomologiaPalitos}.
\end{proof}

To finish, we prove that there are proper induced subgraphs of ${\rm And}(k)'$ whose independence complexes have non-trivial second reduced homology groups. This can be written in terms of the Betti numbers as follows:

\begin{lem}\label{lem:fila4And}
    $\beta_{i,i+4}(I({\rm And}(k)'))\neq0$ for all $i\in\{2,\dots,3k-7\}$.
\end{lem}
\begin{proof}
     By Proposition \ref{prop:HochsterRothVanTuyl}, for every $i\in\{2,\dots,3k-7\}$ it is enough to find an induced subgraph $W_i$ of ${\rm And}(k)'$ with $|V(W_i)|=i+4$ such that $\widetilde{H}_2(\Delta(W_i))\neq0$.

First, we can consider the induced subgraph $W_{3k-7}$ of ${\rm And}(k)'$ on $\{v_1,\dots,v_{3k-3}\}$, with $0< v_1<\cdots<v_{3k-3}=n-1$ and $\{v_1,\dots,v_{3k-3}\}=V({\rm And}(k)')\setminus\{0,4\}$. Now we define $W_i:=W_{i+1}\setminus\{v_{i+5}\}$ for every $i\in\{2,\dots,3k-8\}$, which satisfy $|V(W_i)|=i+4$. Since all these graphs have at least two intervals of length at least 3, we conclude by Proposition \ref{prop:lIntervalosLongMayorIgualt}.\end{proof}

\section{Open problems and future work}

In this paper we introduced the family of graphs ${\rm GA}(t,k)'$, obtained from Generalized Andrásfai graphs by deleting the exterior Hamiltonian cycle. We proved that their edge ideals provide examples with arbitrary regularity and arbitrarily large projective dimension. More precisely, for every $t\ge1$ and $k\ge3$,
\[
\operatorname{reg}(I({\rm GA}(t,k)'))=t+2
\qquad \text{and}\qquad
\operatorname{pd}(I({\rm GA}(t,k)'))=t(k-2).
\]
Moreover, the regularity is only attained at the last step of the resolution and there is only one generator of degree ${\rm reg}(I(G))+{\rm pd}(I(G))$ in the last syzygy module. This extends previous results for complements of cycles and bipartite complements of even cycles.

Using the same methods developed in this paper, one can also treat the edge ideals of the original Generalized Andrásfai graphs ${\rm GA}(t,k)$ (see Figure \ref{fig:BettiAnd6}). In particular, one obtains
\[{\rm reg}(I({\rm GA}(t,k)))=t
\qquad \text{and}\qquad
{\rm pd}(I({\rm GA}(t,k)))=t(k-2)+2.
\]
Since the proofs require no essentially new ideas and would considerably increase the length of the manuscript, we have decided to omit them.

\begin{figure}[htb]
\centering
	\includegraphics[width=13.8cm]{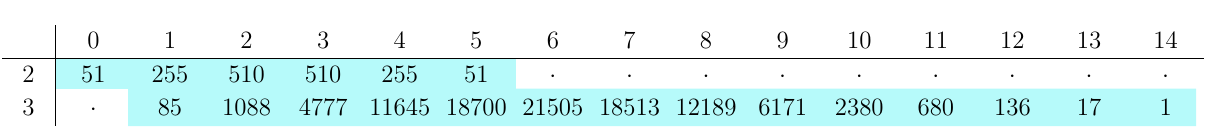} 
  	\caption{Betti diagram of the edge ideal of ${\rm And}(6) = {\rm GA}(3,6)$}
	\label{fig:BettiAnd6}
\end{figure}

Interestingly, one can also prove that ${\rm reg}(I(\GA)) = t$ by combining Lemma \ref{lema:HomologiaLinkSinCiclo} with the following fact: for every $k'> t(k - 1) + k + 1$, $\GA$ is isomorphic to an induced subgraph of ${\rm GA}(t,k')'$ via the map $x \mapsto (t+1)x$ for all $0 \leq x \leq t(k-1)+1$.

\medskip

Our results suggest several natural directions for future research.

For every graph $G$, the existence of induced subgraphs of the form ${\rm GA}(1,k)'$ with $k \geq 3$ characterizes whether ${\rm reg}(I(G)) > 2$. Similarly, for every bipartite graph $G$, the existence of induced subgraphs of the form ${\rm GA}(2,k)'$ with $k \geq 3$ characterizes whether ${\rm reg}(I(G)) > 3$. It would be interesting to describe a natural class of graphs such that for every graph $G$ in this class, the existence of induced subgraphs of the form $\GAsin$ with $k \geq 3$ characterizes whether ${\rm reg}(I(G)) > t + 1$. More generally, it would be interesting to have an answer to the following problem.
\begin{question}
For a fixed integer $r$, characterize the graphs $G$ such that
\[
\operatorname{reg}(I(G))>r.
\]
\end{question}

For $r=2$, the answer is given by Fr\"oberg's Theorem. The answer for $r \geq 3$ might depend on the base field $\mathbb K$ of the polynomial ring we are considering (see \cite{Kat06}). 
For $r=3$, the problem was recently solved by Kanno \cite{Kanno2026}. However, little seems to be known for larger values of $r$.

A second problem concerns the exact shape of the Betti diagrams of $I({\rm GA}(t,k)')$. We completely determined it for $t=3$, while for $t\geq 4$ we obtained upper and lower bounds. Our computations suggest that the Betti diagrams of these ideals follow the pattern  (see Figure \ref{fig:ConjeturaShape}):

\begin{conj}\label{conj:gasin}
For all $t,k\ge3$,  $\beta_{i,i+j}(I(\GAsin))\neq0$ if and only if 
\begin{itemize} \item $j\in\{2,\dots,t+1\}$ and $j-2\le i\le (j-1)(k-2)-1$, or
\item $(i,j)=\bigl(t(k-2),\,t+2\bigr)$.
\end{itemize}
\end{conj}

\begin{figure}[!ht]
 \centering
 	\includegraphics[width=15cm]{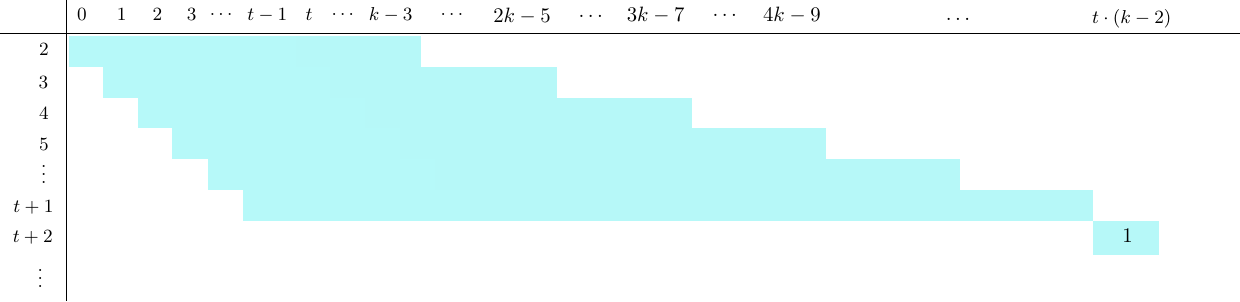}
    \caption{Shape of the Betti diagram of $I(\GAsin)$, according to Conjecture \ref{conj:gasin}}
        \label{fig:ConjeturaShape}
 \end{figure}

Finally, all graph families considered in this paper satisfy the identity
\[
\operatorname{reg}(I(G))+\operatorname{pd}(I(G))=|V(G)|.
\]
This equality also holds for several classical graph families, but it fails in general. It would therefore be interesting to understand when this happens.

\begin{question}
Characterize the graphs $G$ such that
\[
\operatorname{reg}(I(G))+\operatorname{pd}(I(G))=|V(G)|.
\]
\end{question}

For example, it would be natural to investigate whether this holds for large classes of graphs, such as circulant graphs or, more generally, Cayley graphs.

\section*{Acknowledgements}
The three authors were partially supported by the grant PID2022-137283NB-C22 funded by MICIU/AEI/10.13039/501100011033 and by ERDF/EU. Sara Asensio was also supported by European Social Fund Plus, \textit{Programa Operativo de Castilla y León}, and \textit{Junta de Castilla y León} via its \textit{Consejería de Educación}.

\bibliographystyle{plain}
\bibliography{biblio}
    
\end{document}